\documentclass[oneside,reqno,english]{amsart}
\usepackage{lmodern}
\usepackage[T1]{fontenc}
\usepackage[latin9]{inputenc}
\usepackage{babel}
\usepackage{varioref}
\usepackage{float}
\usepackage{amstext}
\usepackage{amsthm}
\usepackage{amssymb}
\usepackage{graphicx}
\usepackage[unicode=true,pdfusetitle,
 bookmarks=true,bookmarksnumbered=false,bookmarksopen=false,
 breaklinks=false,pdfborder={0 0 1},backref=false,colorlinks=false]
 {hyperref}

\makeatletter

\providecommand{\tabularnewline}{\\}
\floatstyle{ruled}
\newfloat{algorithm}{tbp}{loa}
\providecommand{\algorithmname}{Algorithm}
\floatname{algorithm}{\protect\algorithmname}

\numberwithin{equation}{section}
\numberwithin{figure}{section}
\newenvironment{lyxcode}
	{\par\begin{list}{}{
		\setlength{\rightmargin}{\leftmargin}
		\setlength{\listparindent}{0pt}% needed for AMS classes
		\raggedright
		\setlength{\itemsep}{0pt}
		\setlength{\parsep}{0pt}
		\normalfont\ttfamily}%
	 \item[]}
	{\end{list}}
\theoremstyle{plain}
\newtheorem{thm}{\protect\theoremname}[section]
\theoremstyle{plain}
\newtheorem{lyxalgorithm}[thm]{\protect\algorithmname}
\theoremstyle{remark}
\newtheorem{rem}[thm]{\protect\remarkname}
\theoremstyle{plain}
\newtheorem{prop}[thm]{\protect\propositionname}
\theoremstyle{plain}
\newtheorem{lem}[thm]{\protect\lemmaname}
\theoremstyle{plain}
\newtheorem{assumption}[thm]{\protect\assumptionname}

\newcommand{\Nout}{\mbox{\rm N}_{\rm \scriptsize{out}}}

\newcommand{\outdeg}{{\deg}_{\rm out}}
\newcommand{\dom}{\mbox{\rm dom}}

\makeatother

\providecommand{\algorithmname}{Algorithm}
\providecommand{\assumptionname}{Assumption}
\providecommand{\lemmaname}{Lemma}
\providecommand{\propositionname}{Proposition}
\providecommand{\remarkname}{Remark}
\providecommand{\theoremname}{Theorem}

\begin{document}
\begin{lyxcode}
\title[Linear convergence of dual distributed optimization]{Linear convergence of a dual optimization formulation for distributed optimization on directed graphs with unreliable communications} 
\end{lyxcode}

\begin{abstract}
This work builds on our recent work on a distributed optimization
algorithm for graphs with directed unreliable communications. We show
its linear convergence when we take either the proximal of each function
or an affine minorant for when the function is smooth. 
\end{abstract}

\author{C.H. Jeffrey Pang}
\thanks{C.H.J. Pang acknowledges grant R-146-000-265-114 from the Faculty
of Science, National University of Singapore. }
\curraddr{Department of Mathematics\\ 
National University of Singapore\\ 
Block S17 08-11\\ 
10 Lower Kent Ridge Road\\ 
Singapore 119076 }
\email{matpchj@nus.edu.sg}
\date{\today{}}

\maketitle
\tableofcontents{}

\section{Introduction}

Let $G=(V,E)$ be a directed graph. Consider the distributed optimization
problem
\begin{equation}
\begin{array}{c}
\underset{x\in\mathbb{R}^{m}}{\min}\underset{i\in V}{\overset{\phantom{i\in V}}{\sum}}\left[f_{i}(x)+\frac{1}{2}\|x-\bar{x}_{i}\|^{2}\right].\end{array}\label{eq:dist_opt_pblm}
\end{equation}
Here, $f_{i}(\cdot)$ are closed convex functions. The challenge in
distributed optimization is that the communications in the algorithm
need to be along the directed edges in the underlying graph. Ideally,
one would like to solve the problem in \eqref{eq:dist_opt_pblm} without
the quadratic regularization term, but this regularization term shall
be useful for algorithm we describe in this paper. 

If $f_{i}(\cdot)$ are the zero functions and $m=1$, then the minimizer
of \eqref{eq:dist_opt_pblm} is exactly $\frac{1}{|V|}\sum_{i\in V}\bar{x}_{i}$,
which is precisely the averaging consensus problem. Some results on
averaged consensus include \cite{Boyd_distrib_averaging,Distrib_averaging_Dimakis_Kar_Moura_Rabbat_Scaglione}.
Recently, a distributed asynchronous algorithm for averaged consensus
on a directed graph with unreliable communications was designed in
\cite{Bof_Carli_Schenato_2017}, building on the work of \cite{Benezit_Blondel_Thiran_Tsitsilkis_Vetterli_2010_weighted_avr,Vaidya_Hadji_Domin_2011,Hadji_Vaidya_Domin_2016_running_sums}.
The averaged consensus algorithm is useful as a building block for
further distributed optimization algorithms. 

If the averaged consensus algorithm were made to be a building block
of a distributed algorithm, then one has to decide whether the averaged
consensus algorithm has run for sufficiently long. This introduces
communication problems to the algorithm, which affects its effectiveness
and applicability. If the distributed optimization algorithm were
extended from the averaged consensus algorithm instead, then we do
not have this issue. 

\subsection{Distributed dual ascent algorithms for \eqref{eq:dist_opt_pblm}}

We showed in \cite{Pang_Dist_Dyk} that a dual ascent interpretation
of \eqref{eq:dist_opt_pblm} leads to a distributed, asynchronous,
decentralized algorithm with deterministic convergence (while also
showing that the algorithm also works for time-varying graphs) on
undirected graphs. The dual ascent interpretation can be traced to
\cite{Combettes_Dung_Vu_SVAA,Combettes_Dung_Vu_JMAA,Pesquet_Abboud_gang_2017_JMIV}
and perhaps earlier, and in the case where the $f_{i}(\cdot)$ are
all indicator functions of closed convex sets, to Dykstra's algorithm
\cite{Dykstra83,BD86,Han88,Gaffke_Mathar} (see also \cite{Deustch01,BauschkeCombettes11,EsRa11}).
Our proof in \cite{Pang_Dist_Dyk} was adapted from \cite{Gaffke_Mathar},
and also makes use of ideas in \cite{Hundal-Deutsch-97} in order
to show the asynchronous nature of our algorithm. We also developed
this dual ascent interpretation more extensively for undirected graphs
in subsequent works by looking at convergence rates \cite{Pang_rate_D_Dyk}
and for the case when $f_{i}(\cdot)$ are level sets of subdifferentiable
functions \cite{Pang_level_sets_Dyk}. 

An algorithm for the averaged consensus problem on directed graphs
with unreliable communications was recently proposed in \cite{Bof_Carli_Schenato_2017}.
It is natural to ask whether the results for undirected graphs carry
over to the case of directed unreliable communications. In \cite{Pang_dir_1},
we showed that the algorithm of \cite{Bof_Carli_Schenato_2017} can
be generalized to the problem \eqref{eq:dist_opt_pblm} and has a
similar dual ascent intepretation as \cite{Pang_Dist_Dyk}. 

\subsection{Contributions of this paper}

We consider the algorithm in \cite{Pang_dir_1} for the case where
the edges are directed and unreliable. We show that we can apply the
techniques in \cite{Pang_sub_Dyk} (proved for the case of undirected
graphs) to avoid proximal operations on $f_{i}(\cdot)$ by taking
subgradient approximations and using affine minorants. The techniques
in \cite{Pang_rate_D_Dyk} are generalized to give linear convergence
of the dual objective value when the smooth functions $f_{i}(\cdot)$
may be approximated by affine minorants, which leads to the linear
convergence to the primal minimizer. 

\section{Preliminaries: Algorithm description}

In this section, we incorporate \cite{Pang_sub_Dyk,Pang_rate_D_Dyk}
(for the case when some of the functions $f_{i}(\cdot)$ in \eqref{eq:dist_opt_pblm}
are treated as subdifferentiable functions) into \cite{Pang_dir_1},
stating the dual optimization interpretation of \eqref{eq:dist_opt_pblm}
and our algorithm. 

Let $\bar{m}=\frac{1}{|V|}\sum_{i\in V}\bar{x}_{i}$. We have 
\begin{equation}
\sum_{i\in V}\frac{1}{2}\|x-\bar{x}_{i}\|^{2}=\frac{|V|}{2}\|x-\bar{m}\|^{2}+\underbrace{\sum_{i\in V}\bar{x}_{i}^{T}\bar{x}_{i}-|V|\bar{m}^{T}\bar{m}}_{C}.\label{eq:make-C}
\end{equation}
So we can assume that all $\bar{x}_{i}$ in \eqref{eq:dist_opt_pblm}
are equal to $\bar{m}$. (This assumption does not mean that a starting
primal variable needs to be $\bar{m}$.) Let $\{s_{\alpha}\}_{\alpha\in V\cup E}$
be such that 
\begin{equation}
\begin{array}{c}
\underset{\alpha\in V\cup E}{\sum}s_{\alpha}=|V|,\text{ and }s_{\alpha}\begin{cases}
>0 & \text{ for all }\alpha\in V\\
\geq0 & \text{ for all }\alpha\in E.
\end{cases}\end{array}\label{eq:s-sums-to-V}
\end{equation}
Let $\mathbf{x}\in[\mathbb{R}^{m}]^{|V\cup E|}$, and for all $i\in V$,
let $\mathbf{f}_{i}:[\mathbb{R}^{m}]^{|V\cup E|}\to\mathbb{R}\cup\{\infty\}$
be defined as $\mathbf{f}_{i}(\mathbf{x})=f_{i}([\mathbf{x}]_{i})$.
Let the set $F$ be 
\begin{equation}
F:=\big\{\{i,(i,j)\}:(i,j)\in E\big\}\cup\big\{\{i,j\}:(i,j)\in E\big\}.\label{eq:set-F-formula}
\end{equation}
and let the hyperplane $H_{\{\alpha_{1},\alpha_{2}\}}$, where $\{\alpha_{1},\alpha_{2}\}\in F$,
be defined by 
\begin{equation}
H_{\{\alpha_{1},\alpha_{2}\}}:=\{\mathbf{x}\in[\mathbb{R}^{m}]^{|V\cup E|}:[\mathbf{x}]_{\alpha_{1}}=[\mathbf{x}]_{\alpha_{2}}\}.\label{eq:set-H-formula}
\end{equation}
We assume the underlying graph is strongly connected, so the intersection
$\cap_{\beta\in F}H_{\beta}$ is the diagonal set $D$ defined by
\begin{equation}
\cap_{\beta\in F}H_{\beta}=D:=\big\{\mathbf{x}\in[\mathbb{R}^{m}]^{|V\cup E|}:[\mathbf{x}]_{\alpha_{1}}=[\mathbf{x}]_{\alpha_{2}}\text{ for all }\alpha_{1},\alpha_{2}\in V\cup E\big\}.\label{eq:diag-set}
\end{equation}
The primal problem \eqref{eq:dist_opt_pblm} can then be equivalently
written in the product space formulation as 
\begin{equation}
\begin{array}{c}
\underset{\mathbf{x}\in[\mathbb{R}^{m}]^{|V\cup E|}}{\overset{\phantom{\mathbf{x}\in[\mathbb{R}^{m}]^{|V\cup E|}}}{\min}}\underset{\alpha\in V\cup E}{\overset{\phantom{\alpha\in V\cup E}}{\sum}}\frac{s_{\alpha}}{2}\|[\mathbf{x}]_{\alpha}-\bar{m}\|^{2}+\underset{i\in V}{\overset{}{\sum}}\mathbf{f}_{i}(\mathbf{x})+\underset{\beta\in F}{\overset{}{\sum}}\delta_{H_{\beta}}(\mathbf{x})+C,\end{array}\label{eq:2nd-primal}
\end{equation}
where $C$ is as marked in \eqref{eq:make-C}. Any component of an
optimal solution to \eqref{eq:2nd-primal} is an optimal solution
to \eqref{eq:dist_opt_pblm}. The (Fenchel) dual of \eqref{eq:2nd-primal}
can be calculated to be 
\begin{equation}
\sup_{{\mathbf{z}_{\alpha}\in[\mathbb{R}^{m}]^{|V\cup E|}\atop \alpha\in V\cup F}}\frac{|V|}{2}\|\bar{m}\|^{2}-\sum_{i\in V}\mathbf{f}_{i}^{*}(\mathbf{z}_{i})-\sum_{\beta\in F}\delta_{H_{\beta}^{\perp}}(\mathbf{z}_{\beta})-\sum_{\alpha\in V\cup E}\frac{s_{\alpha}}{2}\left\Vert \bar{m}-\frac{1}{s_{\alpha}}\left[\sum_{\alpha_{2}\in V\cup F}\mathbf{z}_{\alpha_{2}}\right]_{\alpha}\right\Vert ^{2}+C.\label{eq:dual-1}
\end{equation}
The case when $s_{\alpha}=1$ for all $\alpha\in V$ and $s_{\alpha}=0$
for all $\alpha\in E$ has been discussed in detail in \cite{Pang_Dist_Dyk,Pang_sub_Dyk,Pang_rate_D_Dyk,Pang_level_sets_Dyk}.
The treatment there implies that there is strong duality between \eqref{eq:2nd-primal}
and \eqref{eq:dual-1}, even if dual optimizers may not exist. We
can define the values $\{x_{\alpha}\}_{\alpha\in V\cup E}$ by 
\begin{equation}
\begin{array}{c}
x_{\alpha}:=\bar{m}-\frac{1}{s_{\alpha}}\Big[\underset{\alpha_{2}\in V\cup F}{\sum}\mathbf{z}_{\alpha_{2}}\Big]_{\alpha},\end{array}\label{eq:x-alpha-form}
\end{equation}
which simplifies the formula in \eqref{eq:dual-1}. As explained in
\cite{Pang_dir_1}, this $x_{\alpha}$ is precisely the primal value
that is being tracked by each vertex or edge $\alpha$. To simplify
discussions, we let 
\[
\mathbf{z}=\{\mathbf{z}_{i}\}_{i\in V}\text{, }\mathbf{x}=\{x_{\alpha}\}_{\alpha\in V\cup E}\text{, and }\text{\ensuremath{\mathbf{s}=\{s_{\alpha}\}_{\alpha\in V\cup E}.}}
\]
Sometimes we may write $[\mathbf{x}]_{\alpha}$ in place of $x_{\alpha}$.
Sometimes we may have $\mathbf{z}$ to mean $\{\mathbf{z}_{\alpha}\}_{\alpha\in V\cup F}$
and not mention $\mathbf{x}$ because of the relationship \eqref{eq:x-alpha-form}.
For convenience, instead of considering \eqref{eq:dual-1}, we may
at times consider 
\begin{equation}
\inf_{{\mathbf{z}_{\alpha}\in[\mathbb{R}^{m}]^{|V\cup E|}\atop \alpha\in V\cup F}}F_{S}(\mathbf{z},\mathbf{s}):=\sum_{i\in V}\mathbf{f}_{i}^{*}(\mathbf{z}_{i})+\sum_{\beta\in F}\delta_{H_{\beta}^{\perp}}(\mathbf{z}_{\beta})+\sum_{\alpha\in V\cup E}\frac{s_{\alpha}}{2}\|x_{\alpha}\|^{2}.\label{eq:dual-2}
\end{equation}
Note that \eqref{eq:dual-1} and \eqref{eq:dual-2} are related by
a sign change and a constant. We partition the vertex set $V$ as
the disjoint union \emph{$V=V_{1}\cup V_{2}$ }so that 
\begin{itemize}
\item $f_{i}(\cdot)$ are proximable functions for all $i\in V_{1}$.
\item \emph{$f_{i}(\cdot)$ }are subdifferentiable functions (i.e., a subgradient
is easy to obtain) such that $\dom(f_{i})=\mathbb{R}^{m}$ for all
$i\in V_{2}$. 
\end{itemize}
In \cite{Pang_dir_1}, we showed that in the case when all the functions
$f_{i}(\cdot)$ are treated as proximable functions (i.e., $V_{2}=\emptyset$),
the algorithm there produces iterates $(\mathbf{z},\mathbf{x},\mathbf{s})$
such that the function values in \eqref{eq:dual-2} are nonincreasing.
For functions in $V_{2}$, the strategy in \cite{Pang_sub_Dyk,Pang_rate_D_Dyk}
is to create approximations $f_{i}^{k}(\cdot)\leq f_{i}(\cdot)$ so
that the conjugates satisfy $[f_{i}^{k}]^{*}(\cdot)\geq f_{i}^{*}(\cdot)$.
Let $\mathbf{f}_{i}^{k}(\cdot)$ be defined in a similar manner as
$\mathbf{f}_{i}(\cdot)$, and
\begin{equation}
\inf_{{\mathbf{z}_{\alpha}\in[\mathbb{R}^{m}]^{|V\cup E|}\atop \alpha\in V\cup F}}\tilde{F}_{S}^{k}(\mathbf{z},\mathbf{x},\mathbf{s}):=\sum_{i\in V_{1}}\mathbf{f}_{i}^{*}(\mathbf{z}_{i})+\sum_{i\in V_{2}}[\mathbf{f}_{i}^{k}]^{*}(\mathbf{z}_{i})+\sum_{\beta\in F}\delta_{H_{\beta}^{\perp}}(\mathbf{z}_{\beta})+\sum_{\alpha\in V\cup E}\frac{s_{\alpha}}{2}\|x_{\alpha}\|^{2}\label{eq:dual-3}
\end{equation}
would be a majorization of the function in \eqref{eq:dual-2}. We
shall prove in Section \ref{sec:Main-result} that solving subproblems
of the form \eqref{eq:dual-3} gives us linear convergence of the
minimal value of \eqref{eq:dual-2} when all the functions $f_{i}(\cdot)$
are smooth. 

Just like in \cite{Bof_Carli_Schenato_2017}, we introduce the variable
$y_{\alpha}$ so that 
\begin{equation}
y_{\alpha}=s_{\alpha}x_{\alpha}\text{ for all }\alpha\in V\cup E.\label{eq:y-equals-s-x-formula}
\end{equation}
With these preliminaries, we present Algorithm \vref{alg:dir-alg}.
Operations $A$ and $B$ in Algorithm \ref{alg:Op_ABC} are the same
as in \cite{Bof_Carli_Schenato_2017,Pang_dir_1}, but Operation $C$
is now modified from \cite{Pang_dir_1} to take into account the setup
in \cite{Pang_sub_Dyk,Pang_rate_D_Dyk}.

\begin{algorithm}[!h]
\begin{lyxalgorithm}
\label{alg:dir-alg} (Main algorithm) We have the following algorithm. 

$\quad$Start with $y_{\alpha}^{0}$ such that $\frac{1}{|V|}\sum_{i\in V}y_{i}^{0}=\bar{m}$,
and $y_{\alpha}^{0}=0$ for all $\alpha\in E$. 

$\quad$Start with $s_{\alpha}^{0}=0$ for all $\alpha\in E$.

$\quad$Start with $s_{i}^{0}=1$, $\sigma_{i,y}^{0}=0$ and $\sigma_{i,s}^{0}=0$
for all $i\in V$. 

$\quad$For all $i\in V_{2}$, let $f_{i}^{0}:\mathbb{R}^{m}\to\mathbb{R}$
and $[\mathbf{z}_{i}^{0}]_{i}$ be

$\quad$$\quad$such that $f_{i}^{0}(\cdot)$ is affine with gradient
$[\mathbf{z}_{i}^{0}]_{i}$ and $f_{i}^{0}(\cdot)\leq f_{i}(\cdot)$,

$\quad$$\quad$and let $y_{i}^{0}\leftarrow y_{i}^{0}-[\mathbf{z}_{i}^{0}]_{i}$.

$\quad$For all $i\in V_{1}$, start with $[\mathbf{z}_{i}^{0}]_{i}=0$. 

$\quad$Start with $\rho_{(i,j),y}^{0}=0$ and $\rho_{(i,j),s}^{0}=0$
for all $(i,j)\in E$. 

$\quad$For $k=1,\dots$

$\quad$$\quad$\% Carry all data from last iteration.

$\quad$$\quad$$y_{\alpha}^{k}=y_{\alpha}^{k-1}$ and $s_{\alpha}^{k}=s_{\alpha}^{k-1}$
for all $\alpha\in V\cup E$, and $f_{i}^{k}(\cdot)=f_{i}^{k-1}(\cdot)$
for all $i\in V_{2}$

$\quad$$\quad$$\sigma_{i,y}^{k}=\sigma_{i,y}^{k-1}$, $\sigma_{i,s}^{k}=\sigma_{i,s}^{k-1}$
and $[\mathbf{z}_{i}^{k}]_{i}=[\mathbf{z}_{i}^{k-1}]_{i}$ for all
$i\in V$

$\quad$$\quad$$\rho_{(i,j),y}^{k}=\rho_{(i,j),y}^{k-1}$ and $\rho_{(i,j),s}^{k}=\rho_{(i,j),s}^{k-1}$
for all $(i,j)\in E$

$\quad$$\quad$Perform operation A, B and/or C in Algorithm \ref{alg:Op_ABC}.

$\quad$end for
\begin{lyxalgorithm}
\label{alg:Op_ABC}(Operations $A$, $B$ and $C$) We describe operations
$A$, $B$ and $C$:

01$\quad$\textbf{$A$ (Node $i$ sends data to all out-neighbors) }

02$\quad$$\quad$Choose a node $i\in V$.

03$\quad$$\quad$$y_{i}^{k}=y_{i}^{k}/(\outdeg(i)+1)$; $s_{i}^{k}:=s_{i}^{k}/(\outdeg(i)+1)$

04$\quad$$\quad$$\sigma_{i,y}^{k}=\sigma_{i,y}^{k}+y_{i}^{k}$;
$\sigma_{i,s}^{k}=\sigma_{i,s}^{k}+s_{i}^{k}$.

05$\quad$\textbf{$B$ (Node $j$ receives data from $i$) }

06$\quad$$\quad$Choose edge $(i,j)\in E$ so that $j$ receives
data along $(i,j)$.

07$\quad$$\quad$$y_{j}^{k}=y_{j}^{k}+\sigma_{i,y}^{k}-\rho_{(i,j),y}^{k}$;
$s_{j}^{k}=s_{j}^{k}+\sigma_{i,s}^{k}-\rho_{(i,j),s}^{k}$

08$\quad$$\quad$$\rho_{(i,j),y}^{k}=\sigma_{i,y}^{k}$; $\rho_{(i,j),s}^{k}=\sigma_{i,s}^{k}$

09$\quad$\textbf{$C$ (Update $y_{j}$ and $[\mathbf{z}_{j}]_{j}$
by minimizing dual function)}

10$\quad$$\quad$Choose a node $j\in V$. 

11$\quad$$\quad$$x_{temp}=\frac{1}{s_{j}^{k}}(y_{j}^{k}+[\mathbf{z}_{j}^{k}]_{j})$

12$\quad$$\quad$If $j\in V_{1}$ (i.e., \textbf{$f_{j}(\cdot)$
to be treated as a proximable function}):

13$\quad$$\quad$$\quad$$x_{j}^{k}:=\arg\min_{x}\frac{s_{j}^{k}}{2}\|x_{temp}-x\|^{2}+f_{j}(x)$

14$\quad$$\quad$$\quad$$[\mathbf{z}_{j}^{k}]_{j}=s_{j}^{k}(x_{j}^{k}-x_{temp})$

15$\quad$$\quad$else (if $j\in V_{2}$, i.e., \textbf{$f_{j}(\cdot)$
treated as a subdifferentiable function})

16$\quad$$\quad$$\quad$Recall $f_{j}^{k-1}(\cdot)\leq f_{j}(\cdot)$
is an affine approximate from previous iterations. 

17$\quad$$\quad$$\quad$Let $v_{j}^{k-1}\in\partial f_{j}(x_{j}^{k-1})$.

18$\quad$$\quad$$\quad$Define $\tilde{f}_{j}^{k-1}:\mathbb{R}^{m}\to\mathbb{R}$
by $\tilde{f}_{j}^{k-1}(x):=[v_{j}^{k-1}]^{T}(x-x_{j}^{k-1})+f_{j}(x_{j}^{k-1})$, 

19$\quad$$\quad$$\quad$$x_{j}^{k}:=\arg\min_{x}\big[\max\{f_{j}^{k-1},\tilde{f}_{j}^{k-1}\}(x)+\frac{s_{j}^{k}}{2}\|x-x_{temp}\|^{2}\big].$

20$\quad$$\quad$$\quad$Let $[\mathbf{z}_{j}^{k}]_{j}=s_{j}^{k}(x_{j}^{k}-x_{temp})$

21$\quad$$\quad$$\quad$Let $f_{j}^{k}(x):=[\mathbf{z}_{j}^{k}]_{j}^{T}(x-x_{j}^{k})+\max\{f_{j}^{k-1},\tilde{f}_{j}^{k-1}\}(x_{j}^{k})$. 

22$\quad$$\quad$$\quad$(We then have $x_{j}^{k}=\arg\min_{x}\big[f_{j}^{k}(x)+\frac{s_{j}^{k}}{2}\|x-x_{temp}\|^{2}\big]$
and $f_{j}^{k}(\cdot)\leq f_{j}(\cdot)$)

23$\quad$$\quad$end
\end{lyxalgorithm}

\end{lyxalgorithm}

\end{algorithm}
We now give a short explanation of Algorithms \ref{alg:dir-alg} and
\ref{alg:Op_ABC}, explaining what was being done in \cite{Bof_Carli_Schenato_2017}
and \cite{Pang_dir_1}.
\begin{rem}
(Algorithms \ref{alg:dir-alg} and \ref{alg:Op_ABC}, and \cite{Bof_Carli_Schenato_2017})
Operations $A$ and $B$ of Algorithm \ref{alg:Op_ABC} are described
in \cite{Bof_Carli_Schenato_2017}. When operation $A$ is carried
out, node $i$ sends data to all its out-neighbors. In operation $B$,
a node receives data from its in-neighbors. Even if node $j$ does
not receive information from node $i$ immediately, the information
is delayed and not lost. For each $(i,j)\in E$, the variable $y_{(i,j)}\in\mathbb{R}^{m}$
defined by 
\[
y_{(i,j)}:=\sigma_{i,y}-\rho_{(i,j),y}
\]
 to be the data that is sent by node $i$ but not yet received by
node $j$. If all information from a node is eventually received by
all its out-neighbors and $f_{i}(\cdot)$ are zero for all $i\in V$,
then then \cite{Bof_Carli_Schenato_2017} proved that $y_{i}^{k}/s_{i}^{k}$
converges linearly to $\frac{1}{|V|}\sum_{i\in V}\bar{x}_{i}$ for
all $i\in V$. 
\end{rem}

\begin{figure}[h]
\begin{tabular}{|ccccc|}
\hline 
\begin{tabular}{c}
\includegraphics[scale=0.25]{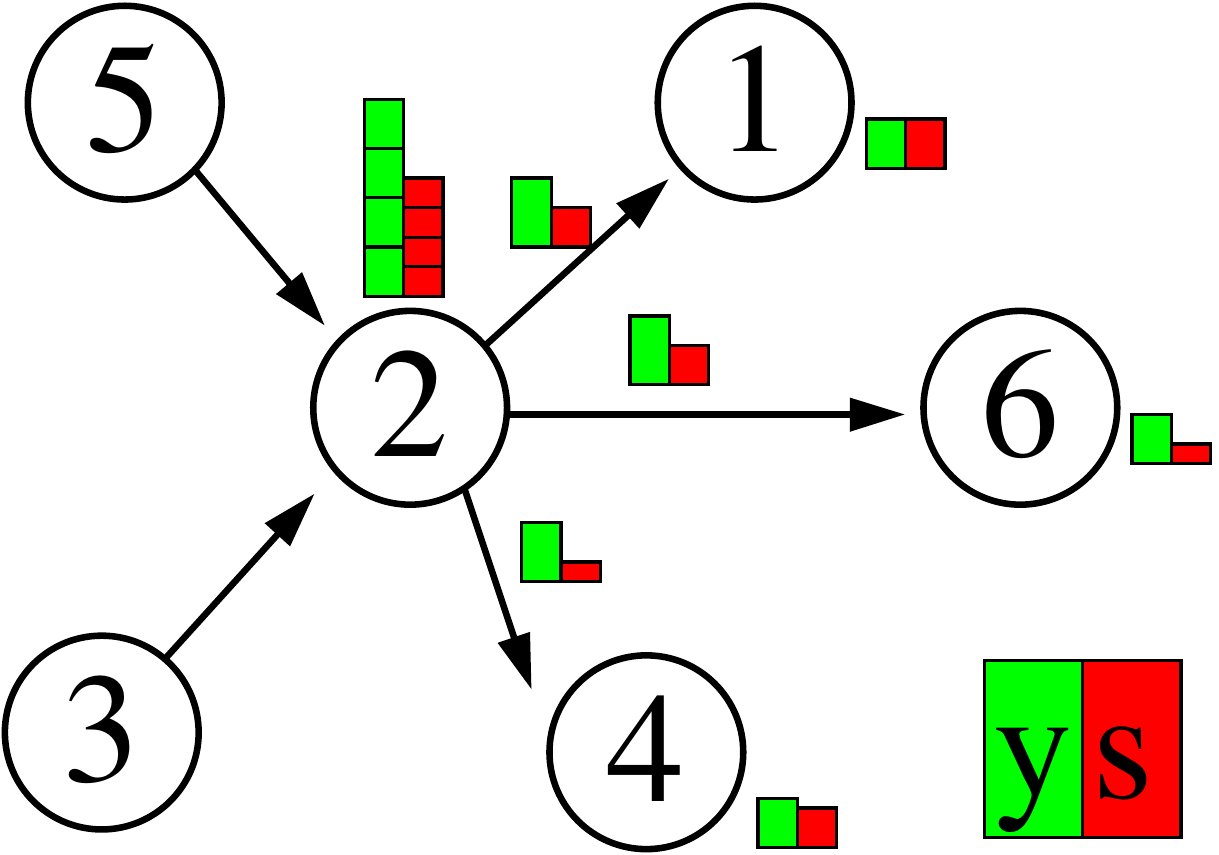}\tabularnewline
\end{tabular} & $\!\!\!\!\!\!\!\!\!\!$%
\begin{tabular}{c}
$\xrightarrow{\scriptsize{\text{Oper. A}}}$\tabularnewline
\end{tabular}$\!\!\!\!\!\!\!\!\!\!\!\!\!\!\!\!\!\!\!$ & %
\begin{tabular}{c}
\includegraphics[scale=0.25]{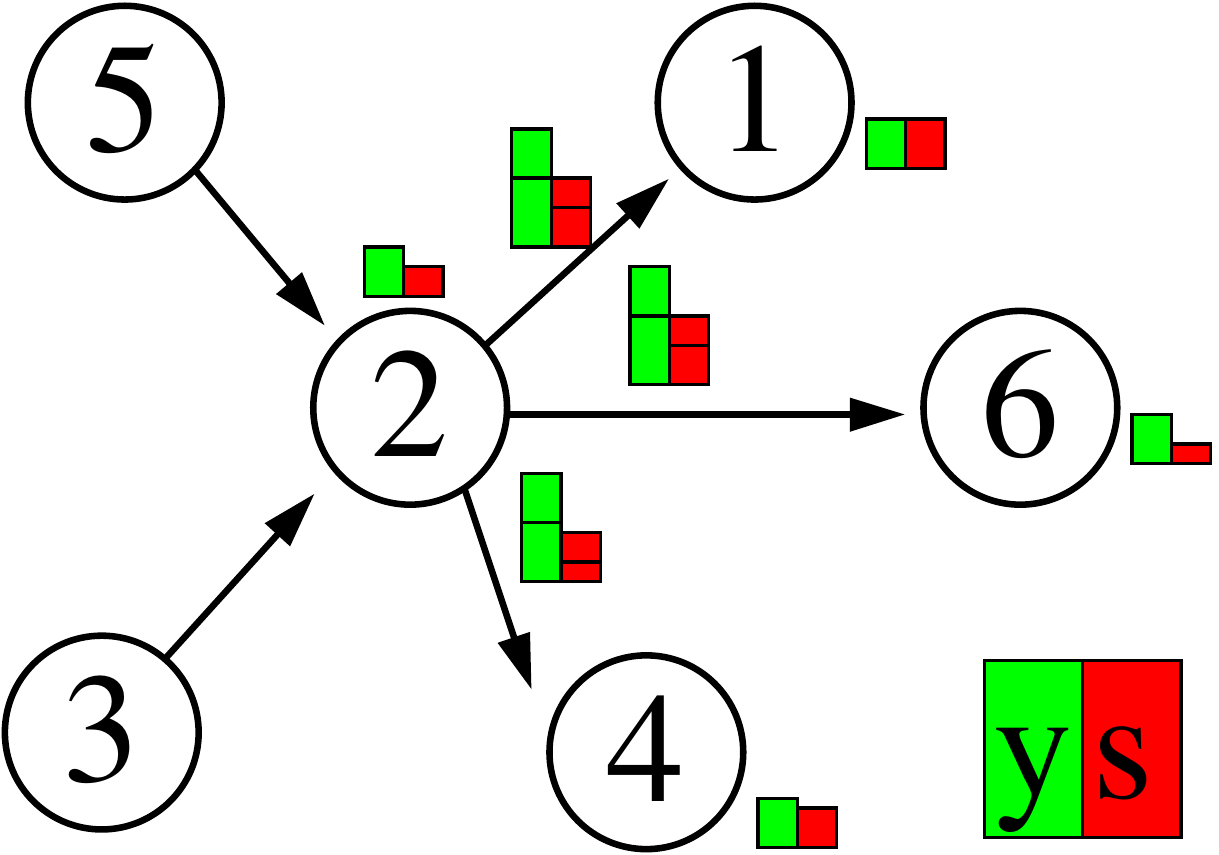}\tabularnewline
\end{tabular} & $\!\!\!\!\!\!\!\!\!$%
\begin{tabular}{c}
$\xrightarrow{\scriptsize{\text{Oper. B}}}$\tabularnewline
\end{tabular}$\!\!\!\!\!\!\!\!\!\!\!\!\!\!\!\!\!\!\!$ & %
\begin{tabular}{c}
\includegraphics[scale=0.25]{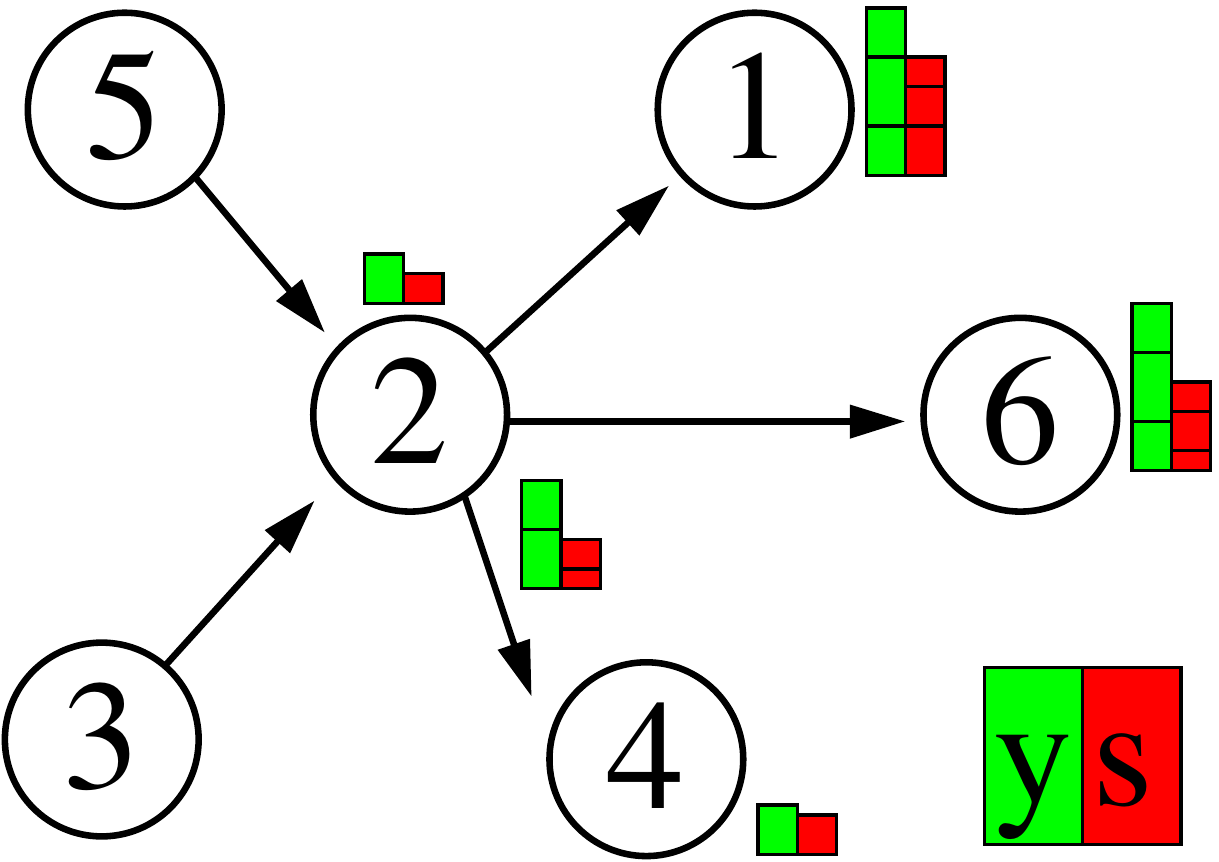}\tabularnewline
\end{tabular}\tabularnewline
\hline 
\multicolumn{5}{|c|}{%
\begin{tabular}{c}
\includegraphics[scale=0.22]{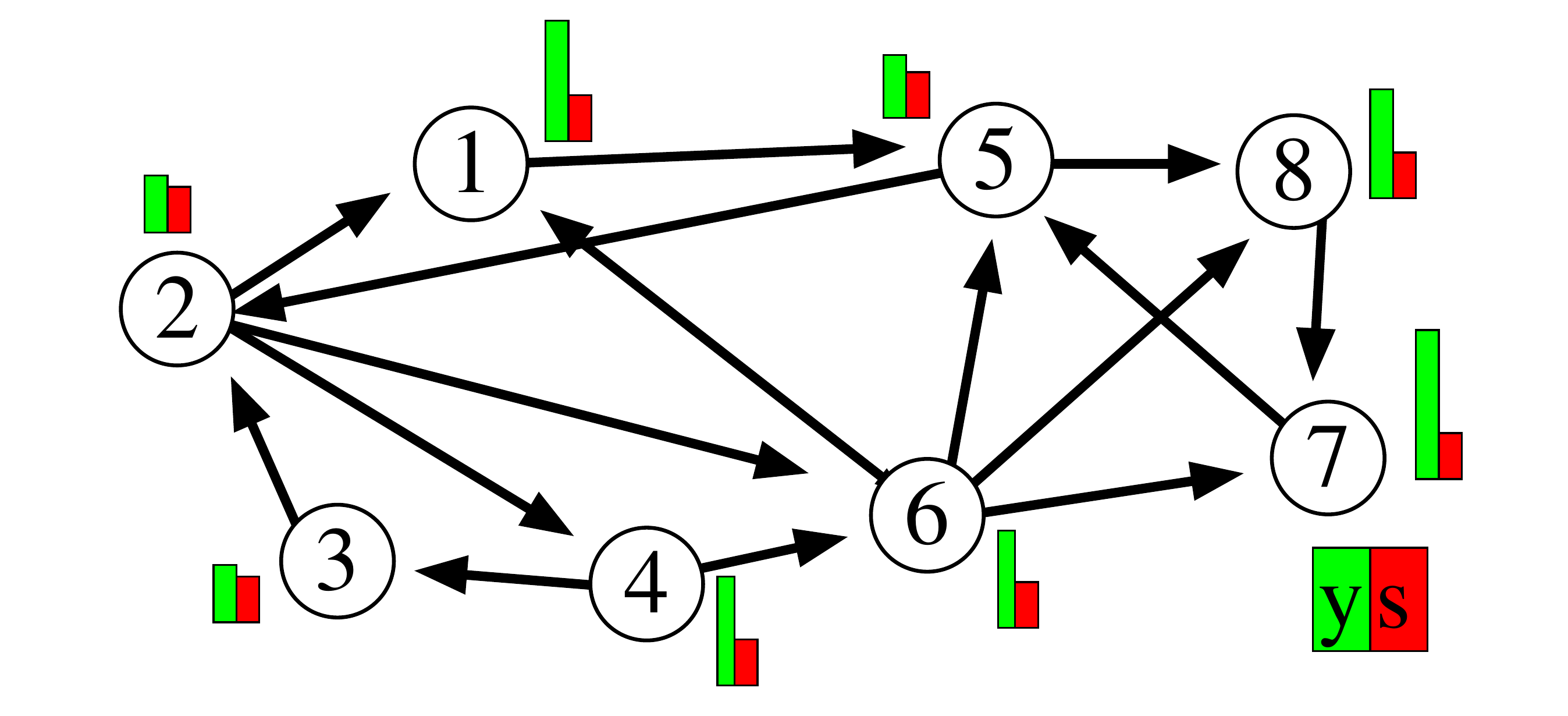}\tabularnewline
\end{tabular}$\!\!\!\!\!\!\!\!\to\!\!\!\!\!\!\!\!\!$%
\begin{tabular}{c}
\includegraphics[scale=0.22]{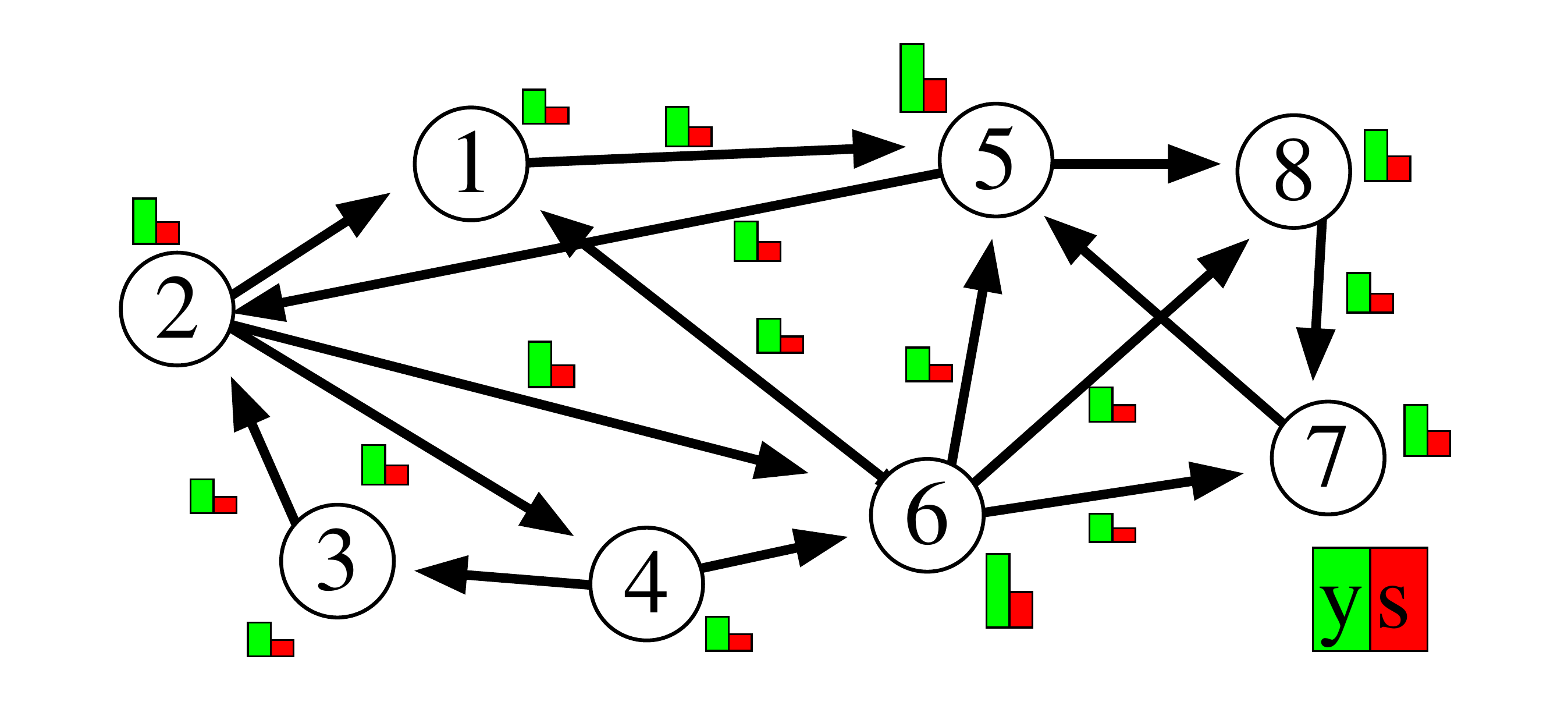}\tabularnewline
\end{tabular}}\tabularnewline
\hline 
\end{tabular}

\caption{The top diagram illustrates Operations $A$ and $B$ in Algorithm
\ref{alg:Op_ABC} due to \cite{Bof_Carli_Schenato_2017}. The bottom
diagram illustrates that in \cite{Bof_Carli_Schenato_2017}, the value
$y_{\alpha}/s_{\alpha}$ converges to the desired average for all
$\alpha\in V\cup E$. }
\end{figure}
\begin{rem}
(Algorithms \ref{alg:dir-alg} and \ref{alg:Op_ABC}, and \cite{Pang_dir_1})
In \cite{Pang_dir_1}, we noticed that the operations $A$ and $B$
reduces the dual objective value in \eqref{eq:dual-3}. We also noticed
that if $f_{i}(\cdot)$ is a proximable function, then lines 13 and
14 in Operation $C$ of Algorithm \ref{alg:Op_ABC} result in a decrease
in the dual objective value \eqref{eq:dual-3}. Lines 16 to 22 of
Operation $C$ incorporates the procedure described in \cite{Pang_sub_Dyk}
to decrease a majorization \eqref{eq:dual-2} of \eqref{eq:dual-3}
for the case when $f_{i}(\cdot)$ is subdifferentiable; this step
allows for a more direct treatment of subdifferentiable functions
$f_{i}(\cdot)$ without having the compute the proximal operations
of lines 13 and 14. Under reasonable conditions, the values $\{x_{\alpha}^{k}\}_{k=0}^{\infty}$
all converge to the optimal primal solution for all $\alpha\in V\cup E$.
We shall show in Theorem \ref{thm:linear-rate} the linear convergence
when all $f_{i}(\cdot)$ are smooth. 
\end{rem}

The following result will be useful for later discussions. 
\begin{prop}
\label{prop:sparsity}(Sparsity) The following results below hold:
\begin{enumerate}
\item If $i\in V$, then $\mathbf{z}_{i}\in[\mathbb{R}^{m}]^{|V\cup E|}$
is such that $[\mathbf{z}_{i}]_{\alpha}=0$ for all $\alpha\in[V\cup E]\backslash\{i\}$. 
\item If $\{\alpha_{1},\alpha_{2}\}\in F$, then $\mathbf{z}_{\{\alpha_{1},\alpha_{2}\}}\in[\mathbb{R}^{m}]^{|V\cup E|}$
is such that $[\mathbf{z}_{\{\alpha_{1},\alpha_{2}\}}]_{\alpha}=0$
for all $\alpha\in[V\cup E]\backslash\{\alpha_{1},\alpha_{2}\}$,
and $[\mathbf{z}_{\{\alpha_{1},\alpha_{2}\}}]_{\alpha_{1}}+[\mathbf{z}_{\{\alpha_{1},\alpha_{2}\}}]_{\alpha_{2}}=0$. 
\end{enumerate}
\end{prop}

\begin{proof}
The proof is elementary and exactly the same as that in \cite{Pang_Dist_Dyk}.
(Part (1) makes use of the fact that $\mathbf{f}_{i}(\cdot)$ depends
on only the $i$-th coordinate of the input, while part (2) makes
use of the fact that $\delta_{H_{\{\alpha_{1},\alpha_{2}\}}}^{*}(\cdot)=\delta_{H_{\{\alpha_{1},\alpha_{2}\}}^{\perp}}(\cdot)$,
and $\delta_{H_{\{\alpha_{1},\alpha_{2}\}}^{\perp}}(\mathbf{z}_{\{\alpha_{1},\alpha_{2}\}})<\infty$
implies the conclusions in (2).)
\end{proof}
The following result is a slight extension of a result in \cite{Pang_sub_Dyk}.
\begin{lem}
\label{lem:subdiff-decrease}\cite{Pang_sub_Dyk} Suppose $f:X\to\mathbb{R}$
is a closed convex subdifferentiable function such that $\dom(f)=X$.
Consider the problem 
\begin{equation}
\begin{array}{c}
\underset{x}{\overset{\phantom{x}}{\min}}\,\,f(x)+\frac{s}{2}\|x-\bar{x}\|^{2},\end{array}\label{eq:lemma-primal}
\end{equation}
which has (Fenchel) dual 
\begin{equation}
\begin{array}{c}
\underset{x}{\overset{\phantom{x}}{\max}}\,\,-f^{*}(z)+\frac{s}{2}\|\bar{x}\|^{2}-\frac{s}{2}\|\frac{1}{s}z-\bar{x}\|^{2}.\end{array}\label{eq:lemma-dual}
\end{equation}
Strong duality is satisfied for this primal dual pair. Let the common
objective value be $v^{*}$. Let $f_{1}:X\to\mathbb{R}$ be an affine
function $f_{1}(x):=a_{1}^{T}x+b_{1}$ such that $f_{1}(\cdot)\leq f(\cdot)$.
We have $f_{1}^{*}(\cdot)\geq f^{*}(\cdot)$. Let $z_{1}$ be the
maximizer of $\max_{z}-f_{1}^{*}(z)+\frac{1}{2}\|\bar{x}\|^{2}-\frac{1}{2}\|z-\bar{x}\|^{2}$,
and let the corresponding solution to the primal problem $\min_{x}f_{1}(x)+\frac{1}{2}\|x-\bar{x}\|^{2}$
be $x_{1}$. Define $\tilde{f}_{1}:X\to\mathbb{R}$ to be an affine
minorant of $f(\cdot)$ at $x_{1}$, i.e., $\tilde{f}_{1}(x)=f(x)+s_{1}^{T}(x-x_{1})$
for some $s_{1}\in\partial f(x_{1})$. Let $x_{2}$ be the minimizer
to the problem 
\begin{equation}
\begin{array}{c}
\underset{x}{\overset{\phantom{x}}{\min}}\,\,[\max\{f_{1}(x),\tilde{f}_{1}(x)\}+\frac{s}{2}\|x-\bar{x}\|^{2}],\end{array}\label{eq:min-max-pblm}
\end{equation}
and let $z_{2}$ be the dual solution. Let $f_{2}:X\to\mathbb{R}$
be the affine function such that the problem \textup{
\[
\begin{array}{c}
\underset{x}{\overset{\phantom{x}}{\min}}f_{2}(x)+\frac{s}{2}\|x-\bar{x}\|^{2}\end{array}
\]
}has the same primal and dual solutions $x_{2}$ and $z_{2}$. Let
\[
\begin{array}{c}
\alpha_{i}=v^{*}-[-f_{i}^{*}(z_{i})+\frac{s}{2}\|\bar{x}\|^{2}-\frac{s}{2}\|\frac{1}{s}z_{i}-\bar{x}\|^{2}]\mbox{ for }i=1,2.\end{array}
\]
One can see that $\alpha_{i}\geq0$, and $\alpha_{i}$ is the measure
of the gap between the estimate of the dual objective value \eqref{eq:lemma-dual}
and its true value $v^{*}$. If $f(\cdot)$ is smooth and $\nabla f(\cdot)$
is Lipschitz with constant $L'$, then 
\begin{equation}
\begin{array}{c}
\frac{1}{4((L'/s)+1)}\left(\frac{\alpha_{2}}{\alpha_{1}}\right)^{2}+\frac{\alpha_{2}}{\alpha_{1}}\leq1.\end{array}\label{eq:lin-dec-subpblm}
\end{equation}
\end{lem}

\begin{proof}
We note that the case where $s=1$ was already treated in \cite{Pang_sub_Dyk}.
For the case where $s\neq1$, we can look at the function $\frac{f(\cdot)}{s}+\frac{1}{2}\|x-\bar{x}\|^{2}$.
Then $\nabla\left(\frac{f}{s}\right)(\cdot)$ has a Lipschitz constant
of $L'/s$, which gives the formula \eqref{eq:lin-dec-subpblm}.
\end{proof}

\section{\label{sec:Main-result}Main result}

In this section, we prove the linear convergence of Algorithm \ref{alg:dir-alg}. 

Throughout this section, we make the following assumption. 
\begin{assumption}
\label{assu:lin-conv-assu}For the problem \eqref{eq:dual-2}, we
make the following assumptions: 

\begin{enumerate}
\item $\{s_{\alpha}\}_{\alpha\in V\cup E}$ satisfies \eqref{eq:s-sums-to-V}.
\item There are constants $s_{\min},s_{\max}>0$ such that for all $\alpha\in V\cup E$,
if $s_{\alpha}>0$, then $s_{\min}\leq s_{\alpha}\leq s_{\max}$. 
\item For all $i\in V$, $f_{i}^{*}(\cdot)$ is strongly convex with modulus
$\sigma>0$, which is equivalent to $\nabla f_{i}(\cdot)$ being Lipschitz
continuous with constant $\frac{1}{\sigma}$. {[}Note that in general,
$f_{i}(\cdot)$ are subdifferentiable for all $i\in V_{2}$, but we
now limit to only smooth $f_{i}(\cdot)$ for our linear convergence
result.{]} 
\end{enumerate}
\end{assumption}

\subsection{\label{subsec:proof-outline}Outline of proof}

We first give an outline of the proof before proceeding with more
technical details. Since Algorithm \ref{alg:dir-alg} is time invariant,
we can assume that we start with the iterates $\{\mathbf{z}^{0},\mathbf{x}^{0},\mathbf{s}^{0}\}$
and the functions $f_{i}^{0}(\cdot)\leq f_{i}(\cdot)$ for all $i\in V_{2}$.
We assume that in the first iteration to get $\{\mathbf{z}^{1},\mathbf{x}^{1},\mathbf{s}^{1}\}$
and the functions $\{f_{i}^{1}(\cdot)\}_{i\in V_{2}}$, operation
$C$ in Algorithm \ref{alg:Op_ABC} is carried out for all $i\in V$.
(We feel that it is simplest to explain in this manner.) Since Operation
$C$ does not change $\mathbf{s}$, we have $\mathbf{s}^{0}=\mathbf{s}^{1}$. 

We also define $\{\mathbf{z}^{+},\mathbf{x}^{+},\mathbf{s}^{0}\}$
to be obtained from $\{\mathbf{z}^{0},\mathbf{x}^{0},\mathbf{s}^{0}\}$
when operation $C$ is conducted for all nodes $i\in V$, but by assuming
the functions $f_{i}(\cdot)$ to be all proximable (i.e., the first
option in Operation $C$ is performed on all nodes). We make use of
Lemma \ref{lem:subdiff-decrease} to relate between $\tilde{F}_{S}^{1}(\mathbf{z}^{1},\mathbf{x}^{1},\mathbf{s}^{1})$
and $F_{S}(\mathbf{z}^{+},\mathbf{x}^{+},\mathbf{s}^{1})$. 

By the case of $s_{i}=1$ for all $i\in V$ and $s_{\alpha}=0$ for
all $\alpha\in E$, we know that there is a finite dual optimal value,
say $F_{S}^{*}$. We assume that there is a constant $K$ such that
for all edges $(i,j)\in E$, there is a $k\in\{1,\dots,K-1\}$ such
that in iterating from $(\mathbf{z}^{k},\mathbf{x}^{k},\mathbf{s}^{k})$
to $(\mathbf{z}^{k+1},\mathbf{x}^{k+1},\mathbf{s}^{k+1})$, operation
$A$ is conducted on node $i$ and operation $B$ is conducted on
edge $(i,j)$ for some $j\in\Nout(i)$. We then show that there is
some constant $c\in(0,1)$ such that $\tilde{F}_{S}^{K}(\mathbf{z}^{K},\mathbf{x}^{K},\mathbf{s}^{K})-F_{S}^{*}\leq c(\tilde{F}_{S}^{0}(\mathbf{z}^{0},\mathbf{x}^{0},\mathbf{s}^{0})-F_{S}^{*})$,
which gives linear convergence to the optimal dual objective value. 

For convenience, we introduce $\mathbf{f}_{\alpha}:[\mathbb{R}^{m}]^{|V\cup E|}\to\mathbb{R}\cup\{\infty\}$
for $\alpha\in E$ defined by $\mathbf{f}_{\alpha}(\cdot)=0$ so that
$\mathbf{f}_{\alpha}^{*}(\cdot)=\delta_{\{0\}}(\cdot)$ for all $\alpha\in E$.
For all $\alpha\in E$, the corresponding variable $\mathbf{z}_{\alpha}$
would satisfy $\delta_{\{0\}}(\mathbf{z}_{\alpha})$ being finite,
which would result in 
\begin{equation}
\mathbf{f}_{\alpha}(\cdot)=0\text{ and }\mathbf{z}_{\alpha}=0\text{ for all }\alpha\in E\text{ throughout.}\label{eq:z-zero-throughout}
\end{equation}
Let $x^{*}\in\mathbb{R}^{m}$ be the optimal solution to \eqref{eq:dist_opt_pblm},
and let $\mathbf{x}^{*}\in[\mathbb{R}^{m}]^{|V\cup E|}$ be such that
all $|V\cup E|$ components of $\mathbf{x}^{*}$ are $x^{*}$. Fenchel
duality gives us \begin{subequations}\label{eq-m:Fenchel-m}
\begin{eqnarray}
\langle\mathbf{x}^{*},\mathbf{z}_{i}^{k}\rangle & \leq & \mathbf{f}_{i}(\mathbf{x}^{*})+\mathbf{f}_{i}^{*}(\mathbf{z}_{i}^{k})\leq\mathbf{f}_{i}(\mathbf{x}^{*})+[\mathbf{f}_{i}^{k}]^{*}(\mathbf{z}_{i}^{k})\text{ for all }i\in V\label{eq:Fenchel-ineq}\\
\text{ and }\langle\mathbf{x}^{*},\mathbf{z}_{\beta}^{k}\rangle & \leq & \delta_{H_{\beta}}(\mathbf{x}^{*})+\delta_{H_{\beta}^{\perp}}(\mathbf{z}_{\beta}^{k})\text{ for all }\beta\in F.\label{eq:Fenchel-ineq-2}
\end{eqnarray}
\end{subequations}For convenience, we define $\mathbf{v}_{H}\in[\mathbb{R}^{m}]^{|V\cup E|}$
as 
\begin{equation}
\begin{array}{c}
\mathbf{v}_{H}:=\underset{\beta\in F}{\overset{\phantom{\beta\in F}}{\sum}}\mathbf{z}_{\beta}.\end{array}\label{eq:v-H}
\end{equation}
Similar to the techniques that we used in \cite{Pang_Dist_Dyk,Pang_dir_1}
that can be traced back to \cite{Gaffke_Mathar}, the duality gap
(in the first line of \eqref{eq:big-formula} below, which is the
optimal value of \eqref{eq:2nd-primal} minus the value of the dual
problem \eqref{eq:dual-1}) satisfies
\begin{eqnarray}
 &  & \sum_{\alpha\in V\cup E}\frac{s_{\alpha}^{k}}{2}\|x^{*}-\bar{m}\|^{2}+\sum_{i\in V}\mathbf{f}_{i}(\mathbf{x}^{*})+\sum_{\beta\in F}\delta_{H_{\beta}}(\mathbf{x}^{*})-\frac{|V|}{2}\|\bar{m}\|^{2}\nonumber \\
 &  & +\sum_{i\in V}[\mathbf{f}_{i}^{k}]^{*}(\mathbf{z}_{i}^{k})+\sum_{\beta\in F}\delta_{H_{\beta}^{\perp}}(\mathbf{z}_{\beta}^{k})+\sum_{\alpha\in V\cup E}\frac{s_{\alpha}^{k}}{2}\left\Vert \bar{m}-\frac{1}{s_{\alpha}^{k}}[\mathbf{v}_{H}^{k}-\mathbf{z}_{\alpha}]_{\alpha}\right\Vert ^{2}\nonumber \\
 & \overset{\eqref{eq-m:Fenchel-m}}{\geq} & \left\langle \mathbf{x}^{*},\sum_{\alpha\in V\cup F}\mathbf{z}_{\alpha}^{k}\right\rangle +\sum_{\alpha\in V\cup E}s_{\alpha}^{k}\left(\frac{1}{2}\|x^{*}-\bar{m}\|^{2}+\frac{1}{2}\left\Vert \bar{m}-\frac{1}{s_{\alpha}^{k}}[\mathbf{v}_{H}^{k}-\mathbf{z}_{\alpha}]_{\alpha}\right\Vert ^{2}-\frac{1}{2}\|\bar{m}\|^{2}\right)\nonumber \\
 & = & \sum_{\alpha\in V\cup E}s_{\alpha}^{k}\left(\left\langle x^{*},\frac{1}{s_{\alpha}^{k}}[\mathbf{v}_{H}^{k}-\mathbf{z}_{\alpha}]_{\alpha}\right\rangle +\frac{1}{2}\|x^{*}\|^{2}-\left\langle x^{*},\bar{m}\right\rangle +\frac{1}{2}\left\Vert \bar{m}-\frac{1}{s_{\alpha}^{k}}[\mathbf{v}_{H}^{k}-\mathbf{z}_{\alpha}]_{\alpha}\right\Vert ^{2}\right)\nonumber \\
 & = & \sum_{\alpha\in V\cup E}\frac{s_{\alpha}^{k}}{2}\left\Vert x^{*}-\left(\bar{m}-\frac{1}{s_{\alpha}^{k}}[\mathbf{v}_{H}^{k}-\mathbf{z}_{\alpha}]_{\alpha}\right)\right\Vert ^{2}\overset{\eqref{eq:x-alpha-form},\eqref{eq:v-H}}{=}\sum_{\alpha\in V\cup E}\frac{s_{\alpha}^{k}}{2}\left\Vert x^{*}-x_{\alpha}^{k}\right\Vert ^{2}.\label{eq:big-formula}
\end{eqnarray}
We will prove that the duality gap converges to zero at a linear rate
in Theorem \ref{thm:linear-rate}. Thus, by \eqref{eq:big-formula}
and Assumption \ref{assu:lin-conv-assu}(2), $\{x_{i}^{k}\}_{k}$
converges to $x^{*}$ at a linear rate for all $i\in V$. 

\subsection{The proof}

We write $\hat{z}^{0}\in\mathbb{R}^{m}$ as 
\begin{eqnarray}
\hat{z}^{0} & := & \begin{array}{c}
-\frac{1}{|V|}\underset{\alpha\in V\cup E}{\overset{\phantom{i\in V}}{\sum}}s_{\alpha}^{0}x_{\alpha}^{0}\overset{\eqref{eq:x-alpha-form}}{=}-\frac{1}{|V|}\underset{\alpha\in V\cup E}{\overset{\phantom{i\in V}}{\sum}}s_{\alpha}^{0}\Big[\bar{m}-\frac{1}{s_{\alpha}^{0}}\big[\underset{\alpha_{2}\in V\cup F}{\sum}\mathbf{z}_{\alpha_{2}}^{0}\big]_{\alpha}\Big]\end{array}\nonumber \\
 & \overset{\scriptsize{\mbox{Prop }\ref{prop:sparsity}}}{=} & \begin{array}{c}
-\frac{1}{|V|}\underset{\alpha\in V\cup E}{\overset{\phantom{i\in V}}{\sum}}\Big[s_{\alpha}^{0}\bar{m}-[\mathbf{z}_{\alpha}^{0}]_{\alpha}\Big]\overset{\eqref{eq:s-sums-to-V}}{=}-\bar{m}+\frac{1}{|V|}\underset{\alpha\in V\cup E}{\overset{\phantom{i\in V}}{\sum}}[\mathbf{z}_{\alpha}^{0}]_{\alpha}.\end{array}\label{eq:z-hat-zero}
\end{eqnarray}
Let $\mathbf{z}^{*}\in[[\mathbb{R}^{m}]^{|V\cup E]}]^{|V\cup F|}$
be an optimal solution to the dual problem $F_{S}(\cdot,\mathbf{s}^{0})$,
and let $\hat{z}^{*}\in\mathbb{R}^{m}$ be defined in a similar manner
to \eqref{eq:z-hat-zero} to be $\hat{z}^{*}=-\bar{m}+\frac{1}{|V|}\sum_{\alpha\in V\cup E}[\mathbf{z}_{\alpha}^{*}]_{\alpha}$.
We have 
\begin{equation}
\begin{array}{c}
|V|(\hat{z}^{0}-\hat{z}^{*})\overset{\eqref{eq:z-hat-zero}}{=}\underset{\alpha\in V\cup E}{\overset{\phantom{\alpha\in V\cup E}}{\sum}}[\mathbf{z}_{\alpha}^{0}-\mathbf{z}_{\alpha}^{*}]_{\alpha}\overset{\eqref{eq:z-zero-throughout}}{=}\underset{\alpha\in V}{\overset{\phantom{\alpha\in V\cup E}}{\sum}}[\mathbf{z}_{\alpha}^{0}-\mathbf{z}_{\alpha}^{*}]_{\alpha}.\end{array}\label{eq:sum-z-to-V}
\end{equation}
We define $\mathbf{e}\in[\mathbb{R}^{m}]^{|V\cup E|}$ so that 
\begin{equation}
\begin{array}{c}
-\hat{z}^{0}+[\mathbf{e}]_{\alpha}=x_{\alpha}^{0}\overset{\scriptsize{\text{\eqref{eq:x-alpha-form},\eqref{eq:v-H},Prop \ref{prop:sparsity}}}}{=}\bar{m}-\frac{1}{s_{\alpha}^{0}}[\mathbf{v}_{H}^{0}+\mathbf{z}_{\alpha}^{0}]_{\alpha}.\end{array}\label{eq:def-e}
\end{equation}
Then $s_{\alpha}^{0}[\mathbf{e}]_{\alpha}\overset{\eqref{eq:def-e}}{=}s_{\alpha}^{0}\bar{m}-[\mathbf{v}_{H}^{0}+\mathbf{z}_{\alpha}^{0}]_{\alpha}+s_{\alpha}^{0}\hat{z}^{0}$
and so 
\begin{equation}
\sum_{\alpha\in V\cup E}s_{\alpha}^{0}[\mathbf{e}]_{\alpha}\overset{\eqref{eq:s-sums-to-V},\eqref{eq:v-H}}{=}|V|\bar{m}-\sum_{\alpha\in V\cup E}[\mathbf{z}_{\alpha}^{0}]_{\alpha}+|V|\hat{z}^{0}\overset{\eqref{eq:z-hat-zero}}{=}0.\label{eq:s-e-i-zero}
\end{equation}
The value $\tilde{F}_{S}^{0}(\mathbf{z}^{0},\mathbf{x}^{0},\mathbf{s}^{0})$
can be written as 
\begin{equation}
\begin{array}{c}
\tilde{F}_{S}^{0}(\mathbf{z}^{0},\mathbf{x}^{0},\mathbf{s}^{0})\overset{\eqref{eq:dual-3},\eqref{eq:def-e}}{=}\underset{\alpha\in V\cup E}{\overset{\phantom{i\in V}}{\sum}}\left[[f_{\alpha}^{0}]^{*}([\mathbf{z}_{\alpha}^{0}]_{\alpha})+\frac{s_{\alpha}^{0}}{2}\|-\hat{z}^{0}+[\mathbf{e}]_{i}\|^{2}\right].\end{array}\label{eq:F-tilde-S}
\end{equation}
Let $\mathbf{z}^{*}$ be a minimizer of $F_{S}(\cdot,\mathbf{s}^{0})$.
The strong convexity of the $f_{i}^{*}(\cdot)$ for all $i\in V$
ensures that \textbf{$\mathbf{z}_{i}^{*}$ }is unique if $i\in V$.
(Though $\mathbf{z}_{\alpha}^{*}$ need not be unique if $\alpha\in F$.)
Since $x^{*}$ is the optimal primal solution, we have $\hat{z}^{*}=-\frac{1}{|V|}\sum_{\alpha\in V\cup E}s_{\alpha}^{0}x^{*}\overset{\eqref{eq:s-sums-to-V}}{=}-x^{*}$.
So the unique solution has the value 
\begin{equation}
\begin{array}{c}
F_{S}^{*}:=F_{S}(\mathbf{z}^{*},\mathbf{s}^{0})\overset{\eqref{eq:dual-2}}{=}\underset{\alpha\in V\cup E}{\overset{\phantom{i\in V}}{\sum}}\left[f_{\alpha}^{*}([\mathbf{z}_{\alpha}^{*}]_{\alpha})+\frac{s_{\alpha}^{0}}{2}\|-\hat{z}^{*}\|^{2}\right].\end{array}\label{eq:def-F-star-s}
\end{equation}

\begin{lem}
\label{lem:gamma-ineq}Suppose Assumption \ref{assu:lin-conv-assu}
holds. Suppose $\{\mathbf{z}_{\alpha}^{0}\}_{\alpha\in V\cup F}$
is a dual variable, and let the derived variables $\mathbf{x}^{0}$,
\textbf{$\mathbf{v}_{H}^{0}$} and $\mathbf{e}$ be as defined in
the above commentary. For all $i\in V$, let $\mathbf{z}_{i}^{+}\in[\mathbb{R}^{m}]^{|V\cup E|}$
be defined so that $[\mathbf{z}_{i}^{+}]_{\alpha}=0$ when $\alpha\neq i$
and 
\begin{equation}
\begin{array}{c}
[\mathbf{z}_{i}^{+}]_{i}=\underset{z\in\mathbb{R}^{m}}{\overset{\phantom{z\in\mathbb{R}^{m}}}{\arg\min}}f_{i}^{*}(z)+\frac{s_{i}^{0}}{2}\|-\hat{z}^{0}+[\mathbf{e}]_{i}+\frac{1}{s_{i}^{0}}[[\mathbf{z}_{i}^{0}]_{i}-z]\|^{2}\mbox{ for all }i\in V.\end{array}\label{eq:z-plus-minimizer}
\end{equation}
For all $i\in V$, let $\tilde{f}_{i}:\mathbb{R}^{m}\to\mathbb{R}$
and $\tilde{\mathbf{f}}_{i}:[\mathbb{R}^{m}]^{|V\cup E|}\to\mathbb{R}$
be related through $\tilde{\mathbf{f}}_{i}(\mathbf{x})=\tilde{f}_{i}([\mathbf{x}]_{i})$.
Assume also that $f_{i}^{0}(\cdot)\leq f_{i}(\cdot)$, which is equivalent
to $[f_{i}^{0}]^{*}(\cdot)\geq f_{i}^{*}(\cdot)$. Let $\tilde{F}_{S}(\cdot,\mathbf{s}):[[\mathbb{R}^{m}]^{|V\cup E|}]^{|V\cup F|}\to\mathbb{R}\cup\{\infty\}$
be defined in a manner similar to \eqref{eq:dual-3} as 
\begin{equation}
\begin{array}{c}
\tilde{F}_{S}^{0}(\mathbf{z},\mathbf{s}):=\underset{i\in V}{\sum}[\tilde{\mathbf{f}}_{i}^{0}]^{*}(\mathbf{z}_{i})+\underset{\beta\in F}{\sum}\delta_{H_{\beta}}^{*}(\mathbf{z}_{\beta})+\underset{\alpha\in V\cup E}{\sum}\frac{s_{\alpha}^{0}}{2}\Big\|\bar{m}-\frac{1}{s_{\alpha}^{0}}[\underset{\alpha_{2}\in V\cup F}{\sum}\mathbf{z}_{\alpha_{2}}]_{\alpha}\Big\|^{2}.\end{array}\label{eq:F-S-tilde}
\end{equation}
Then recalling \eqref{eq:z-zero-throughout}, one can check that 
\begin{align*}
 & \begin{array}{c}
F_{S}(\{\mathbf{z}_{i}\}_{i\in V},\{\mathbf{z}_{\alpha}^{0}\}_{\alpha\in F},\mathbf{s}^{0})=\underset{\alpha\in V\cup E,s_{\alpha}^{0}>0}{\sum}\left[f_{\alpha}^{*}([\mathbf{z}_{\alpha}]_{\alpha})+\frac{s_{\alpha}^{0}}{2}\|-\hat{z}^{0}+[\mathbf{e}]_{\alpha}+\frac{1}{s_{\alpha}^{0}}[\mathbf{z}_{\alpha}^{0}-\mathbf{z}_{\alpha}]_{\alpha}\|^{2}\right]\end{array}\\
\mbox{and} & \begin{array}{c}
\tilde{F}_{S}^{0}(\{\mathbf{z}_{i}\}_{i\in V},\{\mathbf{z}_{\alpha}^{0}\}_{\alpha\in F},\mathbf{s}^{0})=\underset{\alpha\in V\cup E,s_{\alpha}^{0}>0}{\overset{\phantom{\alpha\in V\cup E,s_{\alpha}^{0}>0}}{\sum}}\left[[f_{\alpha}^{0}]^{*}([\mathbf{z}_{\alpha}]_{\alpha})+\frac{s_{\alpha}^{0}}{2}\|-\hat{z}^{0}+[\mathbf{e}]_{\alpha}+\frac{1}{s_{\alpha}^{0}}[\mathbf{z}_{\alpha}^{0}-\mathbf{z}_{\alpha}]_{\alpha}\|^{2}\right].\end{array}
\end{align*}
Let $\mathbf{z}^{*}$ be a minimizer of $F_{S}(\cdot,\mathbf{s}^{0})$,
and let $\mathbf{z}_{\beta}^{+}=\mathbf{z}_{\beta}^{0}$ for all $\beta\in F$.
Then there exists constants $\gamma\in(0,1)$ and $M>0$ such that
if $\mathbf{z}^{0}$ and $\mathbf{z}^{+}$ are related as described,
then 
\begin{eqnarray}
 &  & \begin{array}{c}
F_{S}(\mathbf{z}^{+},\mathbf{s}^{0})-F_{S}(\mathbf{z}^{*},\mathbf{s}^{0})\end{array}\label{eq:lemma-lin-conv-last}\\
 & \leq & \begin{array}{c}
\gamma[F_{S}(\mathbf{z}^{0},\mathbf{s}^{0})-F_{S}(\mathbf{z}^{*},\mathbf{s}^{0})]+M\underset{\alpha\in V\cup E,s_{\alpha}^{0}>0}{\overset{\phantom{V\cup s_{\alpha}^{0}E}}{\sum}}\|[\mathbf{e}]_{\alpha}\|^{2}\end{array}\nonumber \\
 & \leq & \begin{array}{c}
\gamma[\tilde{F}_{S}^{0}(\mathbf{z}^{0},\mathbf{s}^{0})-F_{S}(\mathbf{z}^{*},\mathbf{s}^{0})]+M\underset{\alpha\in V\cup E,s_{\alpha}^{0}>0}{\overset{\phantom{V\cup s_{\alpha}^{0}E}}{\sum}}\|[\mathbf{e}]_{\alpha}\|^{2}.\end{array}\nonumber 
\end{eqnarray}
\end{lem}

\begin{proof}
For this proof, $\mathbf{s}$ will stay as $\mathbf{s}^{0}$ throughout,
so we shall just use $\mathbf{s}$. The second inequality of \eqref{eq:lemma-lin-conv-last}
is obvious from $[f_{i}^{0}]^{*}(\cdot)\geq f_{i}^{*}(\cdot)$ for
all $i\in V_{2}$. We prove the first inequality. By \eqref{eq:z-plus-minimizer}
and Assumption \ref{assu:lin-conv-assu}(3), we have, for all $\alpha\in V\cup E$
such that $s_{\alpha}>0$, 
\begin{eqnarray}
 &  & \begin{array}{c}
f_{\alpha}^{*}([\mathbf{z}_{\alpha}^{0}]_{\alpha})+\frac{s_{\alpha}}{2}\|-\hat{z}^{0}+[\mathbf{e}]_{\alpha}\|^{2}\end{array}\label{eq:f-i-original}\\
 & \overset{\eqref{eq:z-plus-minimizer}}{\geq} & \begin{array}{c}
f_{\alpha}^{*}([\mathbf{z}_{\alpha}^{+}]_{\alpha})+\frac{s_{\alpha}}{2}\|-\hat{z}^{0}+[\mathbf{e}]_{\alpha}+\frac{1}{s_{\alpha}}[\mathbf{z}_{\alpha}^{0}-\mathbf{z}_{\alpha}^{+}]_{\alpha}\|^{2}+\frac{\sigma+(1/s_{\alpha})}{2}\|[\mathbf{z}_{\alpha}^{0}-\mathbf{z}_{\alpha}^{+}]\|^{2}.\end{array}\nonumber \\
 & \overset{\scriptsize{\text{Assu \ref{assu:lin-conv-assu}(2)}}}{\geq} & \begin{array}{c}
f_{\alpha}^{*}([\mathbf{z}_{\alpha}^{+}]_{\alpha})+\frac{s_{\alpha}}{2}\|-\hat{z}^{0}+[\mathbf{e}]_{\alpha}+\frac{1}{s_{\alpha}}[\mathbf{z}_{\alpha}^{0}-\mathbf{z}_{\alpha}^{+}]_{\alpha}\|^{2}+\frac{\sigma+(1/s_{\max})}{2}\|[\mathbf{z}_{\alpha}^{0}-\mathbf{z}_{\alpha}^{+}]\|^{2}.\end{array}\nonumber 
\end{eqnarray}
Also, the optimality condition of \eqref{eq:z-plus-minimizer} implies
that $-\hat{z}^{0}+[\mathbf{e}]_{\alpha}+\frac{1}{s_{\alpha}}[\mathbf{z}_{\alpha}^{0}-\mathbf{z}_{\alpha}^{+}]_{\alpha}\in\partial f_{\alpha}([\mathbf{z}_{\alpha}^{+}]_{\alpha})$.
So together with Assumption \ref{assu:lin-conv-assu}(3), for all
$\alpha\in V\cup E$ such that $s_{\alpha}>0$, we have
\begin{eqnarray}
 &  & \begin{array}{c}
f_{\alpha}^{*}([\mathbf{z}_{\alpha}^{*}]_{\alpha})+\frac{s_{\alpha}}{2}\|-\hat{z}^{*}\|^{2}\end{array}\nonumber \\
 & \overset{\scriptsize{\text{Assu \ref{assu:lin-conv-assu}(3)}}}{\geq} & \begin{array}{c}
f_{\alpha}^{*}([\mathbf{z}_{\alpha}^{+}]_{\alpha})+\frac{s_{\alpha}}{2}\|-\hat{z}^{*}\|^{2}\end{array}\label{eq:f-i-to-opt}\\
 &  & \begin{array}{c}
+\langle-\hat{z}^{0}+[\mathbf{e}]_{\alpha}+\frac{1}{s_{\alpha}}[\mathbf{z}_{\alpha}^{0}-\mathbf{z}_{\alpha}^{+}]_{\alpha},[\mathbf{z}_{\alpha}^{*}-\mathbf{z}_{\alpha}^{+}]_{\alpha}\rangle+\frac{\sigma}{2}\|[\mathbf{z}_{\alpha}^{*}-\mathbf{z}_{\alpha}^{+}]_{\alpha}\|^{2}\end{array}\nonumber \\
 & = & \begin{array}{c}
f_{\alpha}^{*}([\mathbf{z}_{\alpha}^{+}]_{\alpha})+\frac{s_{\alpha}}{2}\|-\hat{z}^{0}+[\mathbf{e}]_{\alpha}+\frac{1}{s_{\alpha}}[\mathbf{z}_{\alpha}^{0}-\mathbf{z}_{\alpha}^{+}]_{\alpha}\|^{2}+\frac{\sigma}{2}\|[\mathbf{z}_{\alpha}^{*}-\mathbf{z}_{\alpha}^{+}]_{\alpha}\|^{2}\end{array}\nonumber \\
 &  & \begin{array}{c}
+\frac{s_{\alpha}}{2}\|-\hat{z}^{*}\|^{2}+\langle-\hat{z}^{0}+[\mathbf{e}]_{\alpha}+\frac{1}{s_{\alpha}}[\mathbf{z}_{\alpha}^{0}-\mathbf{z}_{\alpha}^{+}]_{\alpha},[\mathbf{z}_{\alpha}^{*}-\mathbf{z}_{\alpha}^{+}]_{\alpha}\rangle\end{array}\nonumber \\
 &  & \begin{array}{c}
-(\frac{s_{\alpha}}{2}\|\hat{z}^{0}\|^{2}+\frac{1}{2s_{\alpha}}\|[\mathbf{z}_{\alpha}^{0}-\mathbf{z}_{\alpha}^{+}]_{\alpha}\|^{2}+\frac{s_{\alpha}}{2}\|[\mathbf{e}]_{\alpha}\|^{2}\end{array}\nonumber \\
 &  & \begin{array}{c}
\qquad+\langle-\hat{z}^{0},[\mathbf{z}_{\alpha}^{0}-\mathbf{z}_{\alpha}^{+}]_{\alpha}\rangle+\langle s_{\alpha}[\mathbf{e}]_{\alpha},-\hat{z}^{0}\rangle+\langle[\mathbf{e}]_{\alpha},[\mathbf{z}_{\alpha}^{0}-\mathbf{z}_{\alpha}^{+}]_{\alpha}\rangle).\end{array}\nonumber 
\end{eqnarray}
For the terms not involving $[\mathbf{e}]_{\alpha}$ in the last formula
of \eqref{eq:f-i-to-opt}, we have 
\begin{eqnarray}
 &  & \begin{array}{c}
\underset{\alpha\in V\cup E,s_{\alpha}>0}{\sum}\langle\frac{1}{s_{\alpha}}[\mathbf{z}_{\alpha}^{0}-\mathbf{z}_{\alpha}^{+}]_{\alpha},[\mathbf{z}_{\alpha}^{*}-\mathbf{z}_{\alpha}^{+}]_{\alpha}\rangle\end{array}\nonumber \\
 & \geq & \begin{array}{c}
\underset{\alpha\in V\cup E,s_{\alpha}>0}{\overset{\phantom{i\in V}}{\sum}}\left[-\frac{1}{2\epsilon s_{\alpha}}\|[\mathbf{z}_{\alpha}^{0}-\mathbf{z}_{\alpha}^{+}]_{\alpha}\|^{2}-\frac{\epsilon}{2}\|[\mathbf{z}_{\alpha}^{*}-\mathbf{z}_{\alpha}^{+}]_{\alpha}\|^{2}\right],\end{array}\label{eq:term-1}
\end{eqnarray}
and 
\begin{eqnarray}
 &  & \begin{array}{c}
\underset{\alpha\in V\cup E,s_{\alpha}>0}{\sum}[\langle-\hat{z}^{0},[\mathbf{z}_{\alpha}^{*}-\mathbf{z}_{\alpha}^{+}]_{\alpha}\rangle+\langle\hat{z}^{0},[\mathbf{z}_{\alpha}^{0}-\mathbf{z}_{\alpha}^{+}]_{\alpha}\rangle+\frac{s_{\alpha}}{2}\|-\hat{z}^{*}\|^{2}-\frac{s_{\alpha}}{2}\|\hat{z}^{0}\|^{2}]\end{array}\nonumber \\
 & = & \begin{array}{c}
\underset{\alpha\in V\cup E,s_{\alpha}>0}{\overset{\phantom{i\in V}}{\sum}}[\langle\hat{z}^{0},[\mathbf{z}_{\alpha}^{0}-\mathbf{z}_{\alpha}^{*}]_{\alpha}\rangle+\frac{s_{\alpha}}{2}\|-\hat{z}^{*}\|^{2}-\frac{s_{\alpha}}{2}\|\hat{z}^{0}\|^{2}]\end{array}\label{eq:block-1}\\
 & \overset{\eqref{eq:sum-z-to-V}}{=} & \begin{array}{c}
|V|\langle\hat{z}^{0},\hat{z}^{0}-\hat{z}^{*}\rangle+\frac{|V|}{2}\|\hat{z}^{*}\|^{2}-\frac{|V|}{2}\|\hat{z}^{0}\|^{2}=\frac{|V|}{2}\|\hat{z}^{0}-\hat{z}^{*}\|^{2}\geq0.\end{array}\nonumber 
\end{eqnarray}
For the terms involving $[\mathbf{e}]_{\alpha}$ in the last formula
in \eqref{eq:f-i-to-opt}, we have
\begin{eqnarray}
\begin{array}{c}
\underset{\alpha\in V\cup E,s_{\alpha}>0}{\overset{\phantom{i\in V}}{\sum}}\langle[\mathbf{e}]_{\alpha},[\mathbf{z}_{\alpha}^{*}-\mathbf{z}_{\alpha}^{+}]_{\alpha}\rangle\end{array} & \geq & \begin{array}{c}
\underset{\alpha\in V\cup E,s_{\alpha}>0}{\overset{\phantom{i\in V}}{\sum}}\left[-\frac{1}{2\epsilon}\|[\mathbf{e}]_{\alpha}\|^{2}-\frac{\epsilon}{2}\|[\mathbf{z}_{\alpha}^{*}-\mathbf{z}_{\alpha}^{+}]_{\alpha}\|^{2}\right],\end{array}\nonumber \\
\begin{array}{c}
\underset{\alpha\in V\cup E,s_{\alpha}>0}{\overset{\phantom{i\in V}}{\sum}}\langle s_{\alpha}[\mathbf{e}]_{\alpha},-\hat{z}^{0}\rangle\end{array} & \overset{\eqref{eq:s-e-i-zero}}{=} & \begin{array}{c}
0,\end{array}\label{eq:block-2}\\
\begin{array}{c}
\underset{\alpha\in V\cup E,s_{\alpha}>0}{\overset{\phantom{i\in V}}{\sum}}\langle[\mathbf{e}]_{\alpha},[\mathbf{z}_{\alpha}^{0}-\mathbf{z}_{\alpha}^{+}]_{\alpha}\rangle\end{array} & \geq & \begin{array}{c}
\underset{\alpha\in V\cup E,s_{\alpha}>0}{\overset{\phantom{i\in V}}{\sum}}\left[-\frac{\epsilon}{2}\|[\mathbf{e}]_{\alpha}\|^{2}-\frac{1}{2\epsilon}\|[\mathbf{z}_{\alpha}^{0}-\mathbf{z}_{\alpha}^{+}]_{\alpha}\|^{2}\right].\end{array}\nonumber 
\end{eqnarray}
Summing up the right hand sides of \eqref{eq:term-1}, \eqref{eq:block-1}
and \eqref{eq:block-2} and $\sum_{\alpha\in V\cup E}[\frac{\sigma}{2}\|[\mathbf{z}_{\alpha}^{*}-\mathbf{z}_{\alpha}^{+}]_{\alpha}\|^{2}-\frac{s_{\alpha}}{2}\|[\mathbf{e}]_{\alpha}\|^{2}-\frac{1}{2s_{\alpha}}\|[\mathbf{z}_{\alpha}^{0}-\mathbf{z}_{\alpha}^{+}]_{\alpha}\|^{2}]$
and setting $\epsilon=\sigma/2$ gives 
\begin{eqnarray}
 &  & \begin{array}{c}
\underset{\alpha\in V\cup E,s_{\alpha}>0}{\overset{\phantom{i\in V}}{\sum}}\big[(\frac{\sigma}{2}-\epsilon)\|[\mathbf{z}_{\alpha}^{*}-\mathbf{z}_{\alpha}^{+}]_{\alpha}\|^{2}-(\frac{1}{2\epsilon}+\frac{1}{2\epsilon s_{\alpha}}+\frac{1}{2s_{\alpha}})\|[\mathbf{z}_{\alpha}^{0}-\mathbf{z}_{\alpha}^{+}]_{\alpha}\|^{2}\end{array}\nonumber \\
 &  & \begin{array}{c}
\qquad\qquad\qquad-[\frac{1}{2\epsilon}+\frac{s_{\alpha}}{2}+\frac{\epsilon}{2}]\|[\mathbf{e}]_{\alpha}\|^{2}\big]\end{array}\label{eq:block-3}\\
 & = & \begin{array}{c}
\underset{\alpha\in V\cup E,s_{\alpha}>0}{\overset{\phantom{i\in V}}{\sum}}\big[-(\frac{1}{\sigma}+\frac{1}{\sigma s_{\alpha}}+\frac{1}{2s_{\alpha}})\|[\mathbf{z}_{\alpha}^{0}-\mathbf{z}_{\alpha}^{+}]_{\alpha}\|^{2}-[\frac{1}{\sigma}+\frac{s_{\alpha}}{2}+\frac{\sigma}{4}]\|[\mathbf{e}]_{\alpha}\|^{2}\big]\end{array}\nonumber \\
 & \overset{\scriptsize{\text{Assu \ref{assu:lin-conv-assu}(2)}}}{\geq} & \begin{array}{c}
\underset{\alpha\in V\cup E,s_{\alpha}>0}{\overset{\phantom{i\in V}}{\sum}}\big[-(\frac{1}{\sigma}+\frac{1}{\sigma s_{\min}}+\frac{1}{2s_{\min}})\|[\mathbf{z}_{\alpha}^{0}-\mathbf{z}_{\alpha}^{+}]_{\alpha}\|^{2}-[\frac{1}{\sigma}+\frac{s_{\max}}{2}+\frac{\sigma}{4}]\|[\mathbf{e}]_{\alpha}\|^{2}\big].\end{array}\nonumber 
\end{eqnarray}
Note that $\frac{1}{\sigma}+\frac{1}{\sigma s_{\min}}+\frac{1}{2s_{\min}}=\frac{2s_{\min}+2+\sigma}{2\sigma s_{\min}}$.
Summing the formulas in \eqref{eq:f-i-to-opt} to \eqref{eq:block-3},
we have 
\begin{eqnarray}
 &  & \begin{array}{c}
\underset{\alpha\in V\cup E,s_{\alpha}>0}{\overset{\phantom{i\in V}}{\sum}}\left[f_{\alpha}^{*}([\mathbf{z}_{\alpha}^{*}]_{\alpha})+\frac{1}{2}\|-\hat{z}^{*}\|^{2}\right]\end{array}\nonumber \\
 & \overset{\scriptsize{\text{\eqref{eq:f-i-to-opt} to \eqref{eq:block-3}}}}{\geq} & \begin{array}{c}
\underset{\alpha\in V\cup E,s_{\alpha}>0}{\overset{\phantom{i\in V}}{\sum}}\left[f_{\alpha}^{*}([\mathbf{z}_{\alpha}^{+}]_{\alpha})+\frac{1}{2}\|-\hat{z}^{0}+[\mathbf{e}]_{\alpha}+\frac{1}{s_{\alpha}}[\mathbf{z}_{\alpha}^{0}-\mathbf{z}_{\alpha}^{+}]\|^{2}\right]\end{array}\label{eq:f-i-to-opt-better}\\
 &  & \begin{array}{c}
+\underset{\alpha\in V\cup E,s_{\alpha}>0}{\overset{\phantom{i\in V}}{\sum}}\left[-(\frac{2s_{\min}+2+\sigma}{2\sigma s_{\min}})\|[\mathbf{z}_{\alpha}^{0}-\mathbf{z}_{\alpha}^{+}]_{\alpha}\|^{2}-[\frac{1}{\sigma}+\frac{s_{\max}}{2}+\frac{\sigma}{4}]\|[\mathbf{e}]_{\alpha}\|^{2}\right].\end{array}\nonumber 
\end{eqnarray}
(We say a bit more about \eqref{eq:f-i-to-opt-better}. Recall from
\eqref{eq:s-sums-to-V} that $s_{i}>0$ for all $i\in V$, so $f_{i}^{*}([\mathbf{z}_{i}^{*}]_{i})$
would be part of the sum for all $i\in V$. If $\alpha\in E$, then
\eqref{eq:z-zero-throughout} implies that $f_{\alpha}^{*}(\cdot)=\delta_{\{0\}}(\cdot)$,
so if $f_{\alpha}^{*}([\mathbf{z}_{\alpha}]_{\alpha})$ were to be
finite, then it has to be zero. Thus $f_{\alpha}^{*}([\mathbf{z}_{\alpha}]_{\alpha})$
can be dropped from future sums as needed.) We can then sum up \eqref{eq:f-i-original}
multiplied by $\frac{2s_{\min}+2+\sigma}{\sigma s_{\min}(\sigma+(1/s_{\max}))}$
and \eqref{eq:f-i-to-opt-better} to get 
\begin{eqnarray}
 &  & \begin{array}{c}
\left(\frac{2s_{\min}+2+\sigma}{\sigma s_{\min}(\sigma+(1/s_{\max}))}\right)\underset{\alpha\in V\cup E,s_{\alpha}>0}{\overset{\phantom{i\in V}}{\sum}}[f_{\alpha}^{*}([\mathbf{z}_{\alpha}^{0}]_{\alpha})+\frac{1}{2}\|-\hat{z}^{0}+[\mathbf{e}]_{\alpha}\|^{2}]\end{array}\label{eq:F-S-lin-origin}\\
 &  & \begin{array}{c}
+\underset{\alpha\in V\cup E,s_{\alpha}>0}{\overset{\phantom{i\in V}}{\sum}}\left[f_{\alpha}^{*}([\mathbf{z}_{\alpha}^{*}]_{\alpha})+\frac{1}{2}\|-\hat{z}^{*}\|^{2}\right]\end{array}\nonumber \\
 & \geq & \begin{array}{c}
\left(\frac{2s_{\min}+2+\sigma}{\sigma s_{\min}(\sigma+(1/s_{\max}))}+1\right)\underset{\alpha\in V\cup E,s_{\alpha}>0}{\overset{\phantom{i\in V}}{\sum}}\left[f_{\alpha}^{*}([\mathbf{z}_{\alpha}^{+}]_{\alpha})+\frac{1}{2}\|-\hat{z}^{0}+[\mathbf{e}]_{\alpha}+\frac{1}{s_{\alpha}}[\mathbf{z}_{\alpha}^{0}-\mathbf{z}_{\alpha}^{+}]_{\alpha}\|^{2}\right]\end{array}\nonumber \\
 &  & \begin{array}{c}
-[\frac{1}{\sigma}+\frac{s_{\max}}{2}+\frac{\sigma}{4}]\underset{\alpha\in V\cup E,s_{\alpha}>0}{\sum}\|[\mathbf{e}]_{\alpha}\|^{2}.\end{array}\nonumber 
\end{eqnarray}
Letting $\gamma=\left(\frac{2s_{\min}+2+\sigma}{\sigma s_{\min}(\sigma+(1/s_{\max}))}\right)/\left(\frac{2s_{\min}+2+\sigma}{\sigma s_{\min}(\sigma+(1/s_{\max}))}+1\right)$,
we can rearrange \eqref{eq:F-S-lin-origin} to get\textbf{ }the first
inequality in \eqref{eq:lemma-lin-conv-last} with $M=(1-\gamma)[\frac{1}{\sigma}+\frac{s_{\max}}{2}+\frac{\sigma}{4}]$.
This concludes the proof. 
\end{proof}
Another part of the proof is to show a formula relating the decrease
in objective value to the distance between consecutive iterates. In
order to prove the following result, we need to adopt the convention
throughout the rest of the paper, in addition to \eqref{eq:y-equals-s-x-formula},
that 
\begin{equation}
x_{(i,j)}=x_{i}\text{ whenever }s_{(i,j)}=0.\label{eq:convention}
\end{equation}

\begin{lem}
\label{lem:square-dist-dec}We can assume without loss of generality
that Operation $B$ on the edge $(i,j)$ always follows immediately
after Operation $A$ on the node $i$. Then there is a constant $\gamma_{4}>0$
such that for all consecutive iterates, we have 
\begin{equation}
\begin{array}{c}
\tilde{F}_{S}^{k+1}(\mathbf{z}^{k+1},\mathbf{s}^{k+1})\leq\tilde{F}_{S}^{k}(\mathbf{z}^{k},\mathbf{s}^{k})-\frac{\gamma_{4}}{2}\|\mathbf{x}^{k}-\mathbf{x}^{k+1}\|^{2}.\end{array}\label{eq:decrease-consec}
\end{equation}
\end{lem}

\begin{proof}
We split into three different cases:

\textbf{Case 1: Operation $C$ on node $i$. }

In this case, only $x_{i}$ changes and $\mathbf{s}^{k+1}=\mathbf{s}^{k}$,
so $\|\mathbf{x}^{k}-\mathbf{x}^{k+1}\|^{2}=\|x_{i}^{k}-x_{i}^{k+1}\|^{2}$.
Now, the form of the optimization problem gives 
\begin{eqnarray*}
\begin{array}{c}
\tilde{F}_{S}^{k+1}(\mathbf{z}^{k+1},\mathbf{s}^{k+1})\end{array} & \leq & \begin{array}{c}
\tilde{F}_{S}^{k}(\mathbf{z}^{k},\mathbf{s}^{k})-\frac{s_{i}^{k}}{2}\|\mathbf{x}^{k}-\mathbf{x}^{k+1}\|^{2}\end{array}\\
 & \overset{\scriptsize{\text{Assu \ref{assu:lin-conv-assu}(2)}}}{\leq} & \begin{array}{c}
\tilde{F}_{S}^{k}(\mathbf{z}^{k},\mathbf{s}^{k})-\frac{s_{\min}}{2}\|\mathbf{x}^{k}-\mathbf{x}^{k+1}\|^{2}.\end{array}
\end{eqnarray*}

\textbf{Case 2: Operation $A$ on node $i$. }

We first state an easily checkable identity that would be used often
in this proof: 
\begin{equation}
\begin{array}{c}
s_{a}\|x_{a}\|^{2}+s_{b}\|x_{b}\|^{2}-(s_{a}+s_{b})\|\frac{s_{a}x_{a}+s_{b}x_{b}}{s_{a}+s_{b}}\|^{2}=\frac{s_{a}s_{b}}{s_{a}+s_{b}}\|x_{a}-x_{b}\|^{2}.\end{array}\label{eq:diff-identity}
\end{equation}

There are two further cases.

\textbf{Case 2a: When the data is received by a node $j\in\Nout(i)$.}

We have $x_{j}^{k+1}=\frac{s_{i}^{k+1}x_{i}^{k}+s_{(i,j)}^{k}x_{(i,j)}^{k}+s_{j}^{k}x_{j}^{k}}{s_{i}^{k+1}+s_{(i,j)}^{k}+s_{j}^{k}}$
and $s_{i}^{k+1}=s_{i}^{k}/(\outdeg(i)+1)$. Note also that $s_{j}^{k+1}=s_{i}^{k+1}+s_{(i,j)}^{k}+s_{j}^{k}$.
We first try to show that there is a constant $\gamma_{4}>0$ such
that 
\begin{eqnarray}
 &  & \begin{array}{c}
\frac{s_{i}^{k+1}}{2}\|x_{i}^{k}\|^{2}+\frac{s_{(i,j)}^{k}}{2}\|x_{(i,j)}^{k}\|^{2}+\frac{s_{j}^{k}}{2}\|x_{j}^{k}\|^{2}-\frac{s_{j}^{k+1}}{2}\left\Vert x_{j}^{k+1}\right\Vert ^{2}\end{array}\label{eq:combined-dec-ineq}\\
 & \geq & \begin{array}{c}
\frac{\gamma_{4}}{2}\left(\|x_{(i,j)}^{k+1}-x_{(i,j)}^{k}\|^{2}+(\outdeg(j)+1)\|x_{j}^{k}-x_{j}^{k+1}\|^{2}\right).\end{array}\nonumber 
\end{eqnarray}

Suppose $s_{(i,j)}^{k}>0$. Note that $s_{(i,j)}^{k+1}=0$ in this
case, so $x_{i}^{k}=x_{i}^{k+1}\overset{\eqref{eq:convention}}{=}x_{(i,j)}^{k+1}$.
Then
\begin{eqnarray}
 &  & \begin{array}{c}
\frac{s_{i}^{k+1}}{2}\|x_{i}^{k}\|^{2}+\frac{s_{(i,j)}^{k}}{2}\|x_{(i,j)}^{k}\|^{2}-\frac{s_{i}^{k+1}+s_{(i,j)}^{k}}{2}\left\Vert \frac{s_{i}^{k+1}x_{i}^{k}+s_{(i,j)}^{k}x_{(i,j)}^{k}}{s_{i}^{k+1}+s_{(i,j)}^{k}}\right\Vert ^{2}\end{array}\label{eq:first-dec-ineq}\\
 & \overset{\eqref{eq:diff-identity}}{=} & \begin{array}{c}
\frac{s_{i}^{k+1}s_{(i,j)}^{k}}{2(s_{i}^{k+1}+s_{(i,j)}^{k})}\|x_{i}^{k}-x_{(i,j)}^{k}\|^{2}\overset{\eqref{eq:convention}}{=}\frac{s_{i}^{k+1}s_{(i,j)}^{k}}{2(s_{i}^{k+1}+s_{(i,j)}^{k})}\|x_{(i,j)}^{k+1}-x_{(i,j)}^{k}\|^{2}\end{array}\nonumber \\
 & \overset{\scriptsize{\text{Assu \ref{assu:lin-conv-assu}(2)}}}{\geq} & \begin{array}{c}
\frac{s_{\min}}{4}\|x_{(i,j)}^{k+1}-x_{(i,j)}^{k}\|^{2}.\end{array}\nonumber 
\end{eqnarray}
 Note that if $s_{(i,j)}^{k}=0$, then $x_{(i,j)}^{k+1}\overset{\eqref{eq:convention}}{=}x_{i}^{k+1}=x_{i}^{k}\overset{\eqref{eq:convention}}{=}x_{(i,j)}^{k}$,
so one can check that the left and right hand sides of \eqref{eq:first-dec-ineq}
are both zero, so \eqref{eq:first-dec-ineq} would automatically be
satisfied. 

Next, note that $x_{j}^{k+1}=\frac{1}{s_{i}^{k+1}+s_{(i,j)}^{k}+s_{j}^{k}}[(s_{i}^{k+1}+s_{(i,j)}^{k})\frac{s_{i}^{k+1}x_{i}^{k}+s_{(i,j)}^{k}x_{(i,j)}^{k}}{s_{i}^{k+1}+s_{(i,j)}^{k}}+s_{j}^{k}x_{j}^{k}]$,
which gives 
\begin{equation}
\begin{array}{c}
x_{j}^{k+1}-x_{j}^{k}=\frac{s_{i}^{k+1}+s_{(i,j)}^{k}}{s_{i}^{k+1}+s_{(i,j)}^{k}+s_{j}^{k}}\left[\frac{s_{i}^{k+1}x_{i}^{k}+s_{(i,j)}^{k}x_{(i,j)}^{k}}{s_{i}^{k+1}+s_{(i,j)}^{k}}-x_{j}^{k}\right].\end{array}\label{eq:prep-second-dec}
\end{equation}
Recall $s_{j}^{k+1}=s_{i}^{k+1}+s_{(i,j)}^{k}+s_{j}^{k}$. We have
\begin{align}
 & \begin{array}{c}
\frac{s_{i}^{k+1}+s_{(i,j)}^{k}}{2}\left\Vert \frac{s_{i}^{k+1}x_{i}^{k}+s_{(i,j)}^{k}x_{(i,j)}^{k}}{s_{i}^{k+1}+s_{(i,j)}^{k}}\right\Vert ^{2}+\frac{s_{j}^{k}}{2}\|x_{j}^{k}\|^{2}-\frac{s_{j}^{k+1}}{2}\left\Vert x_{j}^{k+1}\right\Vert ^{2}\end{array}\label{eq:second-dec-ineq}\\
\overset{\eqref{eq:diff-identity}}{=} & \begin{array}{c}
\frac{(s_{i}^{k+1}+s_{(i,j)}^{k})s_{j}^{k}}{2(s_{i}^{k+1}+s_{(i,j)}^{k}+s_{j}^{k})}\left\Vert \frac{s_{i}^{k+1}x_{i}^{k}+s_{(i,j)}^{k}x_{(i,j)}^{k}}{s_{i}^{k+1}+s_{(i,j)}^{k}}-x_{j}^{k}\right\Vert ^{2}\end{array}\nonumber \\
\overset{\eqref{eq:prep-second-dec}}{=} & \begin{array}{c}
\frac{(s_{i}^{k+1}+s_{(i,j)}^{k}+s_{j}^{k})s_{j}^{k}}{2(s_{i}^{k+1}+s_{(i,j)}^{k})}\left\Vert x_{j}^{k+1}-x_{j}^{k}\right\Vert ^{2}\overset{\scriptsize{\text{Assu \ref{assu:lin-conv-assu}(2)}}}{\geq}\frac{s_{\min}}{2}\|x_{j}^{k+1}-x_{j}^{k}\|^{2}.\end{array}\nonumber 
\end{align}
Summing up \eqref{eq:first-dec-ineq} and \eqref{eq:second-dec-ineq}
easily leads to a choice of $\gamma_{4}>0$ so that \eqref{eq:combined-dec-ineq}
holds. 

\textbf{Case 2b: When the data is not received by a node $j\in\Nout(i)$.}

In this case, note that $x_{j}^{k+1}=x_{j}^{k}$, and $x_{(i,j)}^{k+1}=\frac{s_{i}^{k+1}x_{i}^{k}+s_{(i,j)}^{k}x_{(i,j)}^{k}}{s_{i}^{k+1}+s_{(i,j)}^{k}}$.
Just like in \eqref{eq:prep-second-dec}, in the case when $s_{(i,j)}^{k}>0$,
we have $x_{(i,j)}^{k+1}-x_{(i,j)}^{k}=\frac{s_{i}^{k+1}}{s_{i}^{k+1}+s_{(i,j)}^{k}}(x_{i}^{k}-x_{(i,j)}^{k})$,
and so we can reduce $\gamma_{4}$ in \eqref{eq:combined-dec-ineq}
if necessary so that 
\begin{eqnarray}
 &  & \begin{array}{c}
\frac{s_{i}^{k+1}}{2}\|x_{i}^{k}\|^{2}+\frac{s_{(i,j)}^{k}}{2}\|x_{(i,j)}^{k}\|^{2}-\frac{s_{i}^{k+1}+s_{(i,j)}^{k}}{2}\left\Vert \frac{s_{i}^{k+1}x_{i}^{k}+s_{(i,j)}^{k}x_{(i,j)}^{k}}{s_{i}^{k+1}+s_{(i,j)}^{k}}\right\Vert ^{2}\end{array}\nonumber \\
 & \overset{\eqref{eq:diff-identity}}{=} & \begin{array}{c}
\frac{s_{i}^{k+1}s_{(i,j)}^{k}}{2(s_{i}^{k+1}+s_{(i,j)}^{k})}\|x_{i}^{k}-x_{(i,j)}^{k}\|^{2}=\frac{s_{(i,j)}^{k}(s_{i}^{k+1}+s_{(i,j)}^{k})}{2s_{i}^{k+1}}\|x_{(i,j)}^{k+1}-x_{(i,j)}^{k}\|^{2}\end{array}\nonumber \\
 & \overset{\scriptsize{\text{Assu \ref{assu:lin-conv-assu}(2)}}}{\geq} & \begin{array}{c}
\frac{s_{\min}}{2}\|x_{(i,j)}^{k+1}-x_{(i,j)}^{k}\|^{2}\geq\frac{\gamma_{4}}{2}\|x_{(i,j)}^{k+1}-x_{(i,j)}^{k}\|^{2}.\end{array}\label{eq:2a-combined-ineq}
\end{eqnarray}
If $s_{(i,j)}^{k}=0$, then we see that both sides of \eqref{eq:2a-combined-ineq}
are zero, so the choice of $\gamma_{4}$ is irrelevant for \eqref{eq:2a-combined-ineq}
to hold. $\hfill\triangle$

After establishing the inequalities \eqref{eq:combined-dec-ineq}
and \eqref{eq:2a-combined-ineq} for the cases 2a and 2b, we now prove
that \eqref{eq:decrease-consec} holds. If $j$ receives data from
$i$, then by recalling the convention in \eqref{eq:convention},
we have
\begin{equation}
\begin{array}{c}
(\outdeg(j)+1)\|x_{j}^{k}-x_{j}^{k+1}\|^{2}\overset{\eqref{eq:convention}}{\geq}\|x_{j}^{k}-x_{j}^{k+1}\|^{2}+\underset{{\scriptsize{j^{+}\in\Nout(j)}\atop \scriptsize{s_{(j,j^{+})}^{k}=0}}}{\sum}\|x_{(j,j^{+})}^{k}-x_{(j,j^{+})}^{k+1}\|^{2}.\end{array}\label{eq:explain-deg}
\end{equation}
Note that 
\begin{eqnarray}
 &  & \begin{array}{c}
\tilde{F}_{S}^{k}(\mathbf{z}^{k},\mathbf{s}^{k})-\tilde{F}_{S}^{k+1}(\mathbf{z}^{k+1},\mathbf{s}^{k+1})\end{array}\label{eq:diff-duals}\\
 & = & \begin{array}{c}
\underset{{j\in\scriptsize{\Nout(i)}\atop \scriptsize{j\text{ receives data}}}}{\overset{\phantom{{j\in\scriptsize{\Nout(i)}\atop \scriptsize{j\text{ receives data}}}}}{\sum}}\left[\frac{s_{i}^{k+1}}{2}\|x_{i}^{k}\|^{2}+\frac{s_{(i,j)}^{k}}{2}\|x_{(i,j)}^{k}\|^{2}+\frac{s_{j}^{k}}{2}\|x_{j}^{k}\|^{2}-\frac{s_{j}^{k+1}}{2}\left\Vert x_{j}^{k+1}\right\Vert ^{2}\right]\end{array}\nonumber \\
 &  & \begin{array}{c}
+\underset{{j\in\scriptsize{\Nout(i)}\atop \scriptsize{j\text{ does not receive data}}}}{\sum}\left[\frac{s_{i}^{k+1}}{2}\|x_{i}^{k}\|^{2}+\frac{s_{(i,j)}^{k}}{2}\|x_{(i,j)}^{k}\|^{2}-\frac{s_{i}^{k+1}+s_{(i,j)}^{k}}{2}\left\Vert x_{(i,j)}^{k+1}\right\Vert ^{2}\right]\end{array}\nonumber \\
 & \overset{\eqref{eq:combined-dec-ineq},\eqref{eq:2a-combined-ineq}}{\geq} & \begin{array}{c}
\underset{{j\in\scriptsize{\Nout(i)}\atop \scriptsize{j\text{ receives data}}}}{\overset{\phantom{{j\in\scriptsize{\Nout(i)}\atop \scriptsize{j\text{ receives data}}}}}{\sum}}\frac{\gamma_{4}}{2}\left[\|x_{(i,j)}^{k+1}-x_{(i,j)}^{k}\|^{2}+(\outdeg+1)\|x_{j}^{k+1}-x_{j}^{k}\|^{2}\right]\end{array}\nonumber \\
 &  & \begin{array}{c}
+\underset{{j\in\scriptsize{\Nout(i)}\atop \scriptsize{j\text{ does not receive data}}}}{\sum}\frac{\gamma_{4}}{2}\left[\|x_{(i,j)}^{k+1}-x_{(i,j)}^{k}\|^{2}\right]\end{array}\nonumber 
\end{eqnarray}
Combining the inequalities \eqref{eq:diff-duals}, \eqref{eq:explain-deg}
gives us \eqref{eq:decrease-consec} as required. 
\end{proof}
We conclude with the theorem on the linear convergence of Algorithm
\ref{alg:dir-alg}.
\begin{thm}
\label{thm:linear-rate}Suppose Assumption \ref{assu:lin-conv-assu}
is satisfied, and $(\mathbf{z}^{1},\mathbf{x}^{1},\mathbf{s}^{1})$
is obtained from $(\mathbf{z}^{0},\mathbf{x}^{0},\mathbf{s}^{0})$
as described in Subsection \ref{subsec:proof-outline}. Suppose there
is a constant $K>0$ such that for any edge $(i,j)\in E$, there is
some $\bar{k}\in\{1,\dots,K-1\}$ such that in order to obtain $(\mathbf{z}^{\bar{k}+1},\mathbf{x}^{\bar{k}+1},\mathbf{s}^{\bar{k}+1})$
from $(\mathbf{z}^{\bar{k}},\mathbf{x}^{\bar{k}},\mathbf{s}^{\bar{k}})$,
Operation $A$ is conducted on node $i$ followed by Operation $B$
on edge $(i,j)$. Then there is some $\gamma_{5}\in(0,1)$ such that
\[
\tilde{F}_{S}^{K}(\mathbf{z}^{K},\mathbf{s}^{K})-F_{S}^{*}\leq\gamma_{5}[\tilde{F}_{S}^{0}(\mathbf{z}^{0},\mathbf{s}^{0})-F_{S}^{*}].
\]
\end{thm}

\begin{proof}
Let $\mathbf{z}^{+}\in[[\mathbb{R}^{m}]^{|V\cup E|}]^{|V\cup F|}$
satisfy \eqref{eq:z-plus-minimizer}, where $\hat{z}^{0}\overset{\eqref{eq:z-hat-zero}}{=}\frac{1}{|V|}\sum_{\alpha\in V\cup E}[\mathbf{z}_{\alpha}^{0}]_{\alpha}-\bar{m}$
like in Lemma \ref{lem:gamma-ineq}. Lemma \ref{lem:gamma-ineq} shows
that there is a $\gamma\in(0,1)$ such that 
\[
\begin{array}{c}
F_{S}(\mathbf{z}^{+},\mathbf{s}^{0})-F_{S}(\mathbf{z}^{*},\mathbf{s}^{0})\overset{\eqref{eq:lemma-lin-conv-last}}{\leq}\gamma[\tilde{F}_{S}^{0}(\mathbf{z}^{0},\mathbf{s}^{0})-F_{S}(\mathbf{z}^{*},\mathbf{s}^{0})]+M\underset{\alpha\in V\cup E,s_{\alpha}^{0}>0}{\sum}\|\mathbf{e}_{\alpha}\|^{2}.\end{array}
\]
Let $p_{\alpha}=-\hat{z}^{0}+[\mathbf{e}]_{\alpha}+\frac{1}{s_{\alpha}^{0}}[\mathbf{z}_{\alpha}^{0}]_{\alpha}$,
which is the proximal center in the formula \eqref{eq:z-plus-minimizer}.
Lemma \ref{lem:subdiff-decrease} implies that there is a constant
$\gamma_{2}\in[0,1)$ such that, for all $\alpha\in V_{2}$, 
\begin{eqnarray}
 &  & \begin{array}{c}
\big[[f_{\alpha}^{1}]^{*}([\mathbf{z}_{\alpha}^{1}]_{\alpha})+\frac{s_{\alpha}^{0}}{2}\|p_{\alpha}-\frac{1}{s_{\alpha}^{0}}[\mathbf{z}_{\alpha}^{1}]_{\alpha}\|^{2}\big]-[f_{\alpha}^{*}([\mathbf{z}_{\alpha}^{+}]_{\alpha})+\frac{s_{\alpha}^{0}}{2}\|p_{\alpha}-\frac{1}{s_{\alpha}^{0}}[\mathbf{z}_{\alpha}^{+}]_{\alpha}\|^{2}]\end{array}\label{eq:lin-conv-component}\\
 & \leq & \begin{array}{c}
\gamma_{2}\Big[\big[[f_{\alpha}^{0}]^{*}([\mathbf{z}_{\alpha}^{0}]_{\alpha})+\frac{s_{\alpha}^{0}}{2}\|p_{\alpha}-\frac{1}{s_{\alpha}^{0}}[\mathbf{z}_{\alpha}^{0}]_{\alpha}\|^{2}\big]-[f_{\alpha}^{*}([\mathbf{z}_{\alpha}^{+}]_{\alpha})+\frac{s_{\alpha}^{0}}{2}\|p_{\alpha}-\frac{1}{s_{\alpha}^{0}}[\mathbf{z}_{\alpha}^{+}]_{\alpha}\|^{2}]\Big].\end{array}\nonumber 
\end{eqnarray}
Summing \eqref{eq:lin-conv-component} over all $\alpha\in V\cup E$
gives 
\begin{equation}
\tilde{F}_{S}^{1}(\mathbf{z}^{1},\mathbf{s}^{0})-F_{S}(\mathbf{z}^{+},\mathbf{s}^{0})\overset{\eqref{eq:F-S-tilde}}{\leq}\gamma_{2}[\tilde{F}_{S}^{0}(\mathbf{z}^{0},\mathbf{s}^{0})-F_{S}(\mathbf{z}^{+},\mathbf{s}^{0})].\label{eq:gamma-2-ineq}
\end{equation}
Note that $\mathbf{s}^{1}=\mathbf{s}^{0}$. We have 
\begin{eqnarray}
 &  & \tilde{F}_{S}^{1}(\mathbf{z}^{1},\mathbf{s}^{1})-F_{S}(\mathbf{z}^{*},\mathbf{s}^{0})\label{eq:chunk-ineq}\\
 & = & \tilde{F}_{S}^{1}(\mathbf{z}^{1},\mathbf{s}^{1})-F_{S}(\mathbf{z}^{+},\mathbf{s}^{0})+F_{S}(\mathbf{z}^{+},\mathbf{s}^{0})-F_{S}(\mathbf{z}^{*},\mathbf{s}^{0})\nonumber \\
 & \overset{\eqref{eq:gamma-2-ineq}}{\leq} & \gamma_{2}[\tilde{F}_{S}^{0}(\mathbf{z}^{0},\mathbf{s}^{0})-F_{S}(\mathbf{z}^{+},\mathbf{s}^{0})]+F_{S}(\mathbf{z}^{+},\mathbf{s}^{0})-F_{S}(\mathbf{z}^{*},\mathbf{s}^{0})\nonumber \\
 & = & \gamma_{2}[\tilde{F}_{S}^{0}(\mathbf{z}^{0},\mathbf{s}^{0})-F_{S}(\mathbf{z}^{*},\mathbf{s}^{0})]+(1-\gamma_{2})[F_{S}(\mathbf{z}^{+},\mathbf{s}^{0})-F_{S}(\mathbf{z}^{*},\mathbf{s}^{0})]\nonumber \\
 & \overset{\scriptsize{\mbox{Lem }\ref{lem:gamma-ineq}}}{\leq} & \underbrace{[1-(1-\gamma)(1-\gamma_{2})]}_{\gamma_{3}}[\tilde{F}_{S}^{0}(\mathbf{z}^{0},\mathbf{s}^{0})-F_{S}^{*}]+(1-\gamma_{2})M\sum_{\alpha\in V\cup E,s_{\alpha}^{0}>0}\|\mathbf{e}_{\alpha}\|^{2}.\nonumber 
\end{eqnarray}
Since $\gamma<1$ and $\gamma_{2}<1$, the $\gamma_{3}$ as marked
above satisfies $\gamma_{3}\in[0,1)$. We now consider 2 cases. 

\textbf{Case 1: $(1-\gamma_{2})M\sum_{\alpha\in V\cup E,s_{\alpha}^{0}>0}\|\mathbf{e}_{\alpha}\|^{2}\leq\frac{1-\gamma_{3}}{2}[\tilde{F}_{S}^{0}(\mathbf{z}^{0},\mathbf{s}^{0})-F_{S}^{*}]$. }

We make use of the fact that the objective value is nonincreasing
to get 
\begin{equation}
\begin{array}{c}
\tilde{F}_{S}^{K}(\mathbf{z}^{K},\mathbf{s}^{K})-F_{S}^{*}\leq\tilde{F}_{S}^{1}(\mathbf{z}^{1},\mathbf{s}^{1})-F_{S}^{*}\overset{\scriptsize{\text{\eqref{eq:chunk-ineq}, Case 1}}}{\underset{\phantom{\scriptsize{\text{\eqref{eq:chunk-ineq}, Case 1}}}}{\leq}}\frac{1+\gamma_{3}}{2}[\tilde{F}_{S}^{0}(\mathbf{z}^{0},\mathbf{s}^{0})-F_{S}^{*}].\end{array}\label{eq:fin-concl-1}
\end{equation}

\textbf{Case 2: $(1-\gamma_{2})M\sum_{\alpha\in V\cup E,s_{\alpha}^{0}>0}\|\mathbf{e}_{\alpha}\|^{2}\geq\frac{1-\gamma_{3}}{2}[\tilde{F}_{S}^{0}(\mathbf{z}^{0},\mathbf{s}^{0})-F_{S}^{*}]$. }

We recall that $x_{\alpha}^{k}\overset{\eqref{eq:x-alpha-form}}{=}\bar{m}-\frac{1}{s_{\alpha}}\sum_{\alpha_{2}\in V\cup F}[\mathbf{z}_{\alpha_{2}}^{k}]_{\alpha}$.
It is elementary to check from \eqref{eq:z-hat-zero} that $(-\hat{z}^{0},\dots,-\hat{z}^{0})$
is the projection of $\mathbf{x}^{0}$ onto the diagonal set $D$
in the $\|\cdot\|_{s}$ seminorm, i.e.,
\begin{equation}
\begin{array}{c}
-\hat{z}^{0}\overset{\eqref{eq:z-hat-zero}}{=}\underset{x}{\arg\min}\underset{\alpha\in V\cup E}{\overset{\phantom{\alpha\in V\cup E}}{\sum}}s_{\alpha}^{0}\|x-x_{\alpha}^{0}\|^{2}.\end{array}\label{eq:z-hat-is-min}
\end{equation}
We have 
\begin{eqnarray}
 &  & \begin{array}{c}
\underset{\alpha\in V\cup E}{\overset{\phantom{\alpha\in V\cup E}}{\sum}}s_{\alpha}^{0}\|e_{\alpha}\|^{2}\end{array}\overset{\eqref{eq:def-e}}{=}\begin{array}{c}
\underset{\alpha\in V\cup E}{\overset{\phantom{\alpha\in V\cup E}}{\sum}}s_{\alpha}^{0}\|x_{\alpha}^{0}-(-\hat{z}^{0})\|^{2}\end{array}\nonumber \\
 & \overset{\eqref{eq:z-hat-is-min}}{=} & \begin{array}{c}
\underset{x}{\min}\underset{\alpha\in V\cup E}{\overset{\phantom{\alpha\in V\cup E}}{\sum}}s_{\alpha}^{0}\|x_{\alpha}^{0}-x\|^{2}\end{array}\overset{\scriptsize{\text{Assu \eqref{assu:lin-conv-assu}(2)}}}{\leq}\begin{array}{c}
\underset{x}{\min}\underset{\alpha\in V\cup E}{\overset{\phantom{\alpha\in V\cup E}}{\sum}}s_{\max}\|x_{\alpha}^{0}-x\|^{2}\end{array}\nonumber \\
 & = & \begin{array}{c}
s_{\max}d(\mathbf{x}_{\phantom{0}}^{0},D)_{\phantom{2}}^{2}.\end{array}\label{eq:sum-e-to-dist}
\end{eqnarray}
By linear regularity arguments typically used in the method of alternating
projections applied to \eqref{eq:set-H-formula} and \eqref{eq:diag-set},
there is a constant $\kappa>0$ such that $d(\mathbf{x},D)\leq\kappa\max_{\beta\in F}d(\mathbf{x},H_{\beta})$.
Let $\beta^{*}\in F$ be such that $d(\mathbf{x}^{0},D)\leq\kappa d(\mathbf{x}^{0},H_{\beta^{*}})$.
To summarize, 
\begin{eqnarray}
 &  & \begin{array}{c}
\frac{1-\gamma_{3}}{2M(1-\gamma_{2})}[\tilde{F}_{S}^{0}(\mathbf{z}^{0},\mathbf{s}^{0})-F_{S}(\mathbf{z}^{*},\mathbf{s}^{0})]\overset{\scriptsize{\mbox{Case 2}}}{\leq}\underset{\alpha\in V\cup E,s_{\alpha}>0}{\overset{\phantom{i\in V}}{\sum}}\|[\mathbf{e}]_{\alpha}\|^{2}\end{array}\label{eq:lin-reg-arg}\\
 & \overset{\scriptsize{\text{Assu \eqref{assu:lin-conv-assu}(2)}}}{\leq} & \begin{array}{c}
\underset{\alpha\in V\cup E}{\overset{\phantom{i\in V}}{\sum}}\frac{s_{\alpha}^{0}}{s_{\min}}\|[\mathbf{e}]_{\alpha}\|^{2}\overset{\eqref{eq:sum-e-to-dist}}{\leq}\frac{s_{\max}}{s_{\min}}d(\mathbf{x}^{0},D)^{2}\leq\frac{s_{\max}\kappa^{2}}{s_{\min}}d(\mathbf{x}^{0},H_{\beta^{*}})^{2}.\end{array}\nonumber 
\end{eqnarray}

The hyperplane $H_{\beta^{*}}$ can either be of the two cases, $\{i,(i,j)\}$
or $\{i,j\}$, where $(i,j)\in E$. In both cases, suppose in iteration
$\bar{k}$, Operations $A$ is performed on node $i$ transmits and
node $j$ receives the data in iteration, and that $\bar{k}+1\leq K$. 

\textbf{Claim:} There is some $(\mathbf{z}^{\bar{k}+0.5},\mathbf{x}^{\bar{k}+0.5},\mathbf{s}^{\bar{k}+0.5})$
such that the $\gamma_{4}>0$ in Lemma \ref{lem:square-dist-dec}
can be adjusted if necessary so that 
\begin{align}
 & \begin{array}{c}
\tilde{F}_{S}^{\bar{k}+1}(\mathbf{z}^{\bar{k}+1},\mathbf{s}^{\bar{k}+1})\leq F_{S}^{\bar{k}}(\mathbf{z}^{\bar{k}+0.5},\mathbf{s}^{\bar{k}+0.5})\leq\tilde{F}_{S}^{\bar{k}}(\mathbf{z}^{\bar{k}},\mathbf{s}^{\bar{k}})-\frac{\gamma_{4}}{2}\|\mathbf{x}^{\bar{k}}-\mathbf{x}^{\bar{k}+0.5}\|^{2},\end{array}\label{eq:dec-fn-val}\\
\text{ and } & \begin{array}{c}
d(\mathbf{x}^{\bar{k}},H_{\beta^{*}})\leq\frac{s_{\min}+s_{\max}}{2s_{\min}}\|\mathbf{x}^{\bar{k}}-\mathbf{x}^{\bar{k}+0.5}\|.\end{array}\label{eq:distance-bdd}
\end{align}
 We split into two cases 

\textbf{Case 2a: $\beta^{*}=\{i,(i,j)\}$}

We can take $\mathbf{x}^{\bar{k}+0.5}$ to be $\mathbf{x}^{\bar{k}+1}$.
Then \eqref{eq:dec-fn-val} follows from Lemma \ref{lem:square-dist-dec}.
Since $s_{(i,j)}^{\bar{k}+1}=0$, by \eqref{eq:convention}, we have
$x_{i}^{\bar{k}+1}=x_{(i,j)}^{\bar{k}+1}$, so $\mathbf{x}^{\bar{k}+1}\in H_{\beta^{*}}$.
This means that \eqref{eq:distance-bdd} holds even if $\frac{s_{\min}+s_{\max}}{2s_{\min}}$
were replaced by the constant $1$. 

\textbf{Case 2b: $\beta^{*}=\{i,j\}$}

Define $\mathbf{x}^{\bar{k}+0.5}$ and $\mathbf{s}^{\bar{k}+0.5}$
by 
\begin{equation}
x_{\alpha'}^{\bar{k}+0.5}=\begin{cases}
\frac{s_{i}^{\bar{k}+1}x_{i}^{\bar{k}}+s_{j}^{\bar{k}}x_{j}^{\bar{k}}}{s_{i}^{\bar{k}+1}+s_{j}^{\bar{k}}} & \text{ if }\alpha'=j\\
x_{\alpha'}^{\bar{k}} & \text{ otherwise,}
\end{cases}\text{ and }s_{\alpha'}^{\bar{k}+0.5}=\begin{cases}
s_{i}^{\bar{k}}-s_{i}^{\bar{k}+1} & \text{ if }\alpha'=i\\
s_{j}^{\bar{k}}+s_{i}^{\bar{k}+1} & \text{ if }\alpha'=j\\
s_{\alpha'}^{\bar{k}} & \text{\,otherwise. }
\end{cases}\label{eq:k-and-half-step}
\end{equation}
We first show \eqref{eq:dec-fn-val}. From how $\mathbf{z}^{\bar{k}+0.5}$
and $\mathbf{x}^{\bar{k}+0.5}$ are defined, we have
\begin{eqnarray}
 &  & \begin{array}{c}
\tilde{F}_{S}^{\bar{k}}(\mathbf{z}^{\bar{k}},\mathbf{s}^{\bar{k}})-\tilde{F}_{S}^{\bar{k}}(\mathbf{z}^{\bar{k}+0.5},\mathbf{s}^{\bar{k}+0.5})\end{array}\label{eq:decrease-for-k-bar-half}\\
 & \overset{\eqref{eq:F-S-tilde}}{=} & \begin{array}{c}
\frac{s_{i}^{\bar{k}+1}}{2}\|x_{i}^{\bar{k}}\|^{2}+\frac{s_{j}^{\bar{k}}}{2}\|x_{j}^{\bar{k}}\|^{2}-\frac{s_{i}^{\bar{k}+1}+s_{j}^{\bar{k}}}{2}\|\frac{s_{i}^{\bar{k}+1}x_{i}^{\bar{k}}+s_{j}^{\bar{k}}x_{j}^{\bar{k}}}{s_{i}^{\bar{k}+1}+s_{j}^{\bar{k}}}\|^{2}\end{array}\nonumber \\
 & \overset{\eqref{eq:diff-identity}}{=} & \begin{array}{c}
\frac{s_{i}^{\bar{k}+1}s_{j}^{\bar{k}}}{2(s_{i}^{\bar{k}+1}+s_{j}^{\bar{k}})}\|x_{i}^{\bar{k}}-x_{j}^{\bar{k}}\|^{2}\overset{\eqref{eq:k-and-half-step}}{=}\frac{s_{j}^{\bar{k}}(s_{i}^{\bar{k}+1}+s_{j}^{\bar{k}})}{2s_{i}^{\bar{k}+1}}\|x_{j}^{\bar{k}+0.5}-x_{j}^{\bar{k}}\|^{2}\end{array}\nonumber \\
 & \overset{\scriptsize{\text{Assu \ref{assu:lin-conv-assu}(2)}}}{\geq} & \begin{array}{c}
\frac{s_{\min}}{2}\|x_{j}^{\bar{k}+0.5}-x_{j}^{\bar{k}}\|^{2}\overset{\eqref{eq:k-and-half-step}}{=}\frac{s_{\min}}{2}\|\mathbf{x}^{\bar{k}+0.5}-\mathbf{x}^{\bar{k}}\|^{2}.\end{array}\nonumber 
\end{eqnarray}
We now show the first inequality in \eqref{eq:dec-fn-val}. Recall
\eqref{eq:y-equals-s-x-formula}, and let $d_{i}=\outdeg(i)+1$. The
formulas in \eqref{eq:k-and-half-step} can be equivalently written
as 
\begin{equation}
y_{\alpha'}^{\bar{k}+0.5}=\begin{cases}
\left(1-\frac{1}{d_{i}}\right)y_{i}^{\bar{k}} & \text{ if }\alpha'=i\\
y_{j}^{\bar{k}}+\frac{1}{d_{i}}y_{i}^{\bar{k}} & \text{ if }\alpha'=j\\
x_{\alpha'}^{\bar{k}} & \text{ otherwise,}
\end{cases}\text{ and }s_{\alpha'}^{\bar{k}+0.5}=\begin{cases}
\left(1-\frac{1}{d_{i}}\right)s_{i}^{\bar{k}} & \text{ if }\alpha'=i\\
s_{j}^{\bar{k}}+\frac{1}{d_{i}}s_{i}^{\bar{k}} & \text{ if }\alpha'=j\\
s_{\alpha'}^{\bar{k}} & \text{\,otherwise. }
\end{cases}\label{eq:k-and-half-step-y}
\end{equation}
One way to interpret \eqref{eq:k-and-half-step-y} is that the change
from $(\mathbf{z}^{\bar{k}},\mathbf{x}^{\bar{k}},\mathbf{s}^{\bar{k}})$
to $(\mathbf{z}^{\bar{k}+0.5},\mathbf{x}^{\bar{k}+0.5},\mathbf{s}^{\bar{k}+0.5})$
is a transfer of mass from $i$ to $j$. One can see that the change
from $(\mathbf{z}^{\bar{k}+0.5},\mathbf{x}^{\bar{k}+0.5},\mathbf{s}^{\bar{k}+0.5})$
to $(\mathbf{z}^{\bar{k}+1},\mathbf{x}^{\bar{k}+1},\mathbf{s}^{\bar{k}+1})$
involves a transfer of mass from $(i,j)$ to $j$, from $i$ to $(i,j')$
for all other $j'\in\Nout(i)\backslash\{j\}$, as well as possibly
from $(i,j')$ to $j'$ for $j'\in\Nout(i)\backslash\{j\}$ if the
corresponding Operation $B$ were carried out. As we saw in the derivation
in \eqref{eq:decrease-for-k-bar-half}, each transfer of mass reduces
the dual objective value, which will give the first inequality in
\eqref{eq:dec-fn-val}. 

To see \eqref{eq:distance-bdd}, note that 
\begin{align*}
 & \begin{array}{c}
d(\mathbf{x}^{\bar{k}},H_{\beta^{*}})^{2}\overset{\eqref{eq:set-H-formula}}{=}\|x_{i}^{\bar{k}}-\frac{1}{2}(x_{i}^{\bar{k}}+x_{j}^{\bar{k}})\|^{2}+\|x_{j}^{\bar{k}}-\frac{1}{2}(x_{i}^{\bar{k}}+x_{j}^{\bar{k}})\|^{2}=\frac{1}{2}\|x_{i}^{\bar{k}}-x_{j}^{\bar{k}}\|^{2}\end{array}\\
\overset{\eqref{eq:k-and-half-step}}{=} & \begin{array}{c}
\frac{s_{i}^{\bar{k}+1}+s_{j}^{\bar{k}}}{2s_{i}^{\bar{k}+1}}\|x_{j}^{\bar{k}}-x_{j}^{\bar{k}+0.5}\|^{2}\overset{\scriptsize{\text{Assu \ref{assu:lin-conv-assu}(2)}}}{\leq}\frac{s_{\min}+s_{\max}}{2s_{\min}}\|x_{j}^{\bar{k}}-x_{j}^{\bar{k}+0.5}\|^{2}\end{array}\\
\overset{\eqref{eq:k-and-half-step}}{=} & \begin{array}{c}
\frac{s_{\min}+s_{\max}}{2s_{\min}}\|\mathbf{x}^{\bar{k}}-\mathbf{x}^{\bar{k}+0.5}\|^{2}.\end{array}
\end{align*}
This ends the claim. $\hfill\triangle$

Now, 
\begin{eqnarray}
 &  & \begin{array}{c}
\tilde{F}_{S}^{K}(\mathbf{z}^{K},\mathbf{s}^{K})\leq\tilde{F}_{S}^{\bar{k}}(\mathbf{z}^{\bar{k}+0.5},\mathbf{s}^{\bar{k}+0.5})\end{array}\label{eq:final-1}\\
 & \overset{\scriptsize{\mbox{Lem \ref{lem:square-dist-dec},\eqref{eq:dec-fn-val}}}}{\leq} & \begin{array}{c}
\tilde{F}_{S}^{0}(\mathbf{z}^{0},\mathbf{s}^{0})-\frac{\gamma_{4}}{2}\Big(\|\mathbf{x}^{\bar{k}}-\mathbf{x}^{\bar{k}+0.5}\|^{2}+\underset{k'=0}{\overset{\bar{k}-1}{\sum}}\|\mathbf{x}^{k'}-\mathbf{x}^{k'+1}\|^{2}\Big).\end{array}\nonumber 
\end{eqnarray}
Let $\tilde{x}^{\bar{k}+0.5}\in H_{\beta^{*}}$ be such that $\|\mathbf{x}^{\bar{k}}-\tilde{\mathbf{x}}^{\bar{k}+0.5}\|=d(\mathbf{x}^{\bar{k}},H_{\beta^{*}})$.
Then
\begin{eqnarray}
 &  & \begin{array}{c}
\frac{\gamma_{4}}{2}\Big(\|\mathbf{x}^{\bar{k}}-\mathbf{x}^{\bar{k}+0.5}\|^{2}+\underset{k'=0}{\overset{\bar{k}-1}{\sum}}\|\mathbf{x}^{k'}-\mathbf{x}^{k'+1}\|^{2}\Big)\end{array}\label{eq:final-2}\\
 & \overset{\eqref{eq:distance-bdd}}{\geq} & \begin{array}{c}
\frac{\gamma_{4}s_{\min}}{(s_{\min}+s_{\max})}\Big(\|\mathbf{x}^{\bar{k}}-\tilde{\mathbf{x}}^{\bar{k}+0.5}\|^{2}+\underset{k'=0}{\overset{\bar{k}-1}{\sum}}\|\mathbf{x}^{k'}-\mathbf{x}^{k'+1}\|^{2}\Big)\end{array}\nonumber \\
 & \geq & \begin{array}{c}
\frac{\gamma_{4}s_{\min}}{(\bar{k}+1)(s_{\min}+s_{\max})}\Big\|(\mathbf{x}^{\bar{k}}-\tilde{\mathbf{x}}^{\bar{k}+0.5})+\underset{k'=0}{\overset{\bar{k}-1}{\sum}}(\mathbf{x}^{k'}-\mathbf{x}^{k'+1})\Big\|^{2}\end{array}\nonumber \\
 & \geq & \begin{array}{c}
\frac{\gamma_{4}s_{\min}}{K(s_{\min}+s_{\max})}\|\mathbf{x}^{0}-\tilde{\mathbf{x}}^{\bar{k}+0.5}\|^{2}\overset{\tilde{\mathbf{x}}^{\bar{k}+0.5}\in H_{\beta^{*}}}{\geq}\frac{\gamma_{4}s_{\min}}{K(s_{\min}+s_{\max})}d(\mathbf{x}^{0},H_{\beta^{*}})^{2}\end{array}\nonumber \\
 & \overset{\eqref{eq:lin-reg-arg}}{\geq} & \begin{array}{c}
\frac{\gamma_{4}s_{\min}^{2}(1-\gamma_{3})}{2K(s_{\min}+s_{\max})s_{\max}M(1-\gamma_{2})\kappa^{2}}[\tilde{F}_{S}^{0}(\mathbf{z}^{0},\mathbf{s}^{0})-F_{S}^{*}].\end{array}\nonumber 
\end{eqnarray}
Combining \eqref{eq:final-1} and \eqref{eq:final-2} gives 
\begin{equation}
\begin{array}{c}
\tilde{F}_{S}^{K}(\mathbf{z}^{K},\mathbf{s}^{K})-F_{S}^{*}\overset{\eqref{eq:final-1},\eqref{eq:final-2}}{\leq}\Big(1-\frac{\gamma_{4}s_{\min}^{2}(1-\gamma_{3})}{2K(s_{\min}+s_{\max})s_{\max}M(1-\gamma_{2})\kappa^{2}}\Big)[\tilde{F}_{S}^{0}(\mathbf{z}^{0},\mathbf{s}^{0})-F_{S}^{*}].\end{array}\label{eq:final-concl-2}
\end{equation}
Combining \eqref{eq:fin-concl-1} for case 1 and \eqref{eq:final-concl-2}
gives the conclusion needed.  
\end{proof}

\section{Numerical experiments}

We conduct some simple experiments by looking at the case where $m=6$
and the graph has 6 nodes and contains two cycles, $1\to2\to3\to5\to1$
and $2\to4\to6\to2$. Let $\mathbf{e}$ be \texttt{ones(m,1)}. First,
we find $\{v_{i}\}_{i\in V}$ and $\bar{x}$ such that $\sum_{i\in V}v_{i}+|V|(\mathbf{e}-\bar{x})=0$.
We then find closed convex functions $f_{i}(\cdot)$ such that $v_{i}\in\partial f_{i}(\mathbf{e})$.
It is clear from the KKT conditions that $\mathbf{e}$ is the primal
optimum solution to \eqref{eq:dist_opt_pblm} if $\bar{x}_{i}=\bar{x}$
for all $i\in V$. 

We define $f_{i}(\cdot)$ as functions of the following type:
\begin{itemize}
\item [(F-S)]$f_{i}(x):=\frac{1}{2}x^{T}A_{i}x+b_{i}^{T}x+c_{i}$, where
$A_{i}$ is of the form $vv^{T}+rI$, where $v$ is generated by \texttt{rand(m,1)},
$r$ is generated by \texttt{rand(1)}. $b_{i}$ is chosen to be such
that $v_{i}=\nabla f(\mathbf{e})$, and $c_{i}=0$. 
\end{itemize}
The first and last formulas of \eqref{eq:big-formula} indicate how
fast the primal iterates $\{x_{\alpha}\}_{\alpha\in V\cup E}$ are
converging to the optimal solution $x^{*}$, and we call these values
the ``duality gap'' and the ``norms squared $s$-sum'' Figure
\ref{fig:fig} shows a plot of the results obtained by a random experiment
where we perform 1000 iterations of the smooth case. We conduct two
different experiments: one for when all functions are treated to be
in $V_{1}$ (called the prox case) and one when all functions are
treated to be in $V_{2}$ (called the subdifferentiable case). We
observe linear convergence for both cases. 

\begin{figure}[!h]
\includegraphics[scale=0.3]{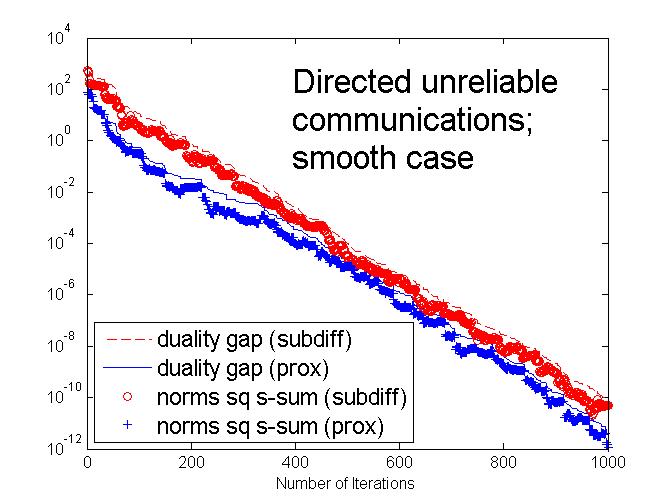}

\caption{\label{fig:fig}Plots of the formulas in \eqref{eq:big-formula}}

\end{figure}
\bibliographystyle{amsalpha}
\bibliography{../refs}

\end{document}